\documentclass[12pt]{amsart}

\usepackage{graphicx, verbatim,amsmath,amssymb, float}
\usepackage{hyperref}
\hypersetup{
    colorlinks=true,
    linkcolor=blue,
    filecolor=magenta,      
    urlcolor=cyan,
}

\usepackage{amsthm}
\usepackage{amsfonts}
\usepackage{cleveref}
\usepackage{comment}
\usepackage{braket}

\newcommand{\tdens}{\tau}
\newcommand{\edens}{\varepsilon}

\newcommand{\calG}{\mathcal{G}}

\newcommand{\calO}{\mathcal{O}}

\newcommand{\be}{\begin{equation}}
\newcommand{\ee}{\end{equation}}

\newtheorem{lemma}{Lemma}

\newtheorem{theorem}[lemma]{Theorem}
\newtheorem{corollary}[lemma]{Corollary}

\newtheorem{proposition}[lemma]{Proposition}

\title{Existence of a symmetric bipodal phase in the edge-triangle model}

\date{\today}

\author{
Joe Neeman
\and Charles Radin
\and Lorenzo Sadun
}
\address{Joe Neeman\\Department of Mathematics\\The University of
  Texas at Austin\\ Austin, TX 78712} \email{joeneeman@gmail.com}
\address{Charles Radin\\Department of Mathematics\\The University of
  Texas at Austin\\ Austin, TX 78712} \email{radin@math.utexas.edu}
\address{Lorenzo Sadun\\Department of Mathematics\\The University of
  Texas at Austin\\ Austin, TX 78712} \email{sadun@math.utexas.edu}

\thanks{
This work was partially supported by the Deutsche Forschungsgemeinschaft (DFG, German Research
Foundation) under Germany's Excellence Strategy – EXC-2047/1 – 390685813,
and by a fellowship from the Alfred P. Sloan Foundation. This material is based upon work supported by the National Science Foundation under Grant Nos. 2145800 and 2204449. Any opinions, findings, and conclusions or recommendations expressed in this material are those of the author(s) and do not necessarily reflect the views of the National Science Foundation.
}

\begin{document}

\begin{abstract}
In the edge-triangle model with edge density close to 1/2 and triangle density below 1/8 we prove that the unique 
entropy-maximizing graphon is symmetric bipodal. We also prove that,
for any edge density $e$ less than $e_0 = (3-\sqrt{3})/6 \approx 0.2113$ and 
triangle density slightly less than $e^3$, the entropy-maximizing graphon is not symmetric bipodal. 
 \end{abstract}

\maketitle

\section{Introduction and Results}\label{sec:Intro}
\subsection{Results}
We study emergent {\em smoothness with respect to change of competing
  constraints} in asymptotically large dense random graphs. More
specifically, we determine and study smooth phases separated by sharp
transitions. We derive a new
phase in the model with sharp constraints on edge and
triangle densities. The phase is ``symmetric bipodal'' and we show how to use
its symmetry to distinguish the phase {\it intrinsically} from other
phases. Unlike in previous work, graphs in this phase are not small perturbations
of Erd\H{o}s-R\'enyi graphs; this requires new techniques, which we develop. 

Let $g(x,y)$ be a graphon, a measurable symmetric function
$g: [0,1]^2 \to [0,1]$. Let 
\begin{equation}\label{3-quantities}
\edens(g) = \iint g(x,y) \, dx \, dy, \qquad 
\tdens(g) = \iiint g(x,y)g(y,z)g(z,x) \, dx \, dy \, dz, 
\end{equation}
and 
\begin{equation}
S(g)=\iint H[g(x,y)]\,dx\, dy, \hbox{ where }
H(u) = -[u \ln(u) + (1-u) \ln(1-u)].
\end{equation}
The integrals $\edens(g)$ and $\tdens(g)$ represent the overall edge and triangle
densities of graphs whose adjacency matrices are close to the graphon
in the ``cut metric'' \cite{Lov},
while $S(g)$ is proportional to the entropy of a random process for generating such graphs. 
We refer to $S(g)$ as the ``entropy of the graphon $g$,'' distinct but
related to the Boltzmann entropy defined in the next paragraph. The quantities
$\edens(g)$, $\tdens(g)$ and $S(g)$ are all invariant under measure-preserving
transformations of $[0,1]$. All statements about uniqueness of graphons should 
be understood to mean ``unique up to measure-preserving transformations 
of $[0,1]$.''

We consider constrained systems, where the edge and triangle densities are constrained
to a vanishingly small tolerance as follows. For each achievable
ordered pair $(e,t)$, we
define the Boltzmann entropy in terms of the partition
function $Z^{n,\delta}_{e,t}$, which is the cardinality of the
set of graphs on $n$ nodes with edge density in the interval
$(e-\delta,e+\delta)$ and triangle density in the interval
$(t-\delta,t+\delta)$. From the partition function  we define the Boltzmann entropy \cite{RS1}
as
\begin{equation}\label{Boltzmann}
B(e,t)=\lim_{\delta \searrow 0}\lim_{n\to
  \infty}\frac{1}{n^2}\ln\big[ Z^{n,\delta}_{e,t}\big].
\end{equation}
Using \cite{CV} we established \cite{RS1} the variational
formula
\begin{equation}\label{variation}
  B(e,t)=\sup_{g\in \calG_{e,t}}S(g),
\end{equation}
where
\begin{equation}\label{constrained}
  \calG_{e,t} = \{ g\, |\, \edens(g)=e, \tdens(g)=t\}.
\end{equation}
\noindent If the supremum of $S(g)$ is attained by a unique graphon $g_{e,t} \in
\calG(e,t)$, then all but exponentially few large graphs with the constrained
edge and triangle densities are described by $g_{e,t}$.
In particular, the density of all possible subgraphs are given by 
integrals involving $g_{e,t}$. 

Our first major result is:
\begin{theorem}\label{main1}
There is an open subset ${\calO}$ of the $(e,t)$ plane,
containing the interval $e=\frac12$,\ $0<t<\frac18$, on which
the unique $S$-maximizing
graphon is  
\begin{equation}\label{symbipodal} g_{e,t}(x,y) = \begin{cases}
    e- (e^3-t)^{1/3},\  &  0 < x,y < \frac12
\hbox{ or } \frac12<x,y<1,\cr
e+(e^3-t)^{1/3},\  & \, x<\frac12< y \hbox{ or }
y < \frac12<x. \end{cases} \end{equation}
\end{theorem}


See Figures \ref{FIG:openset} and \ref{FIG:bipartite3c}.
The following is immediate by inspection.

\begin{corollary}\ $B(e,t)$ and the densities of all subgraphs
are real analytic in $(e,t)$ in ${\calO}$.\end{corollary}
\begin{figure}[ht]
\centering
\includegraphics[angle=0,width=0.4\textwidth]{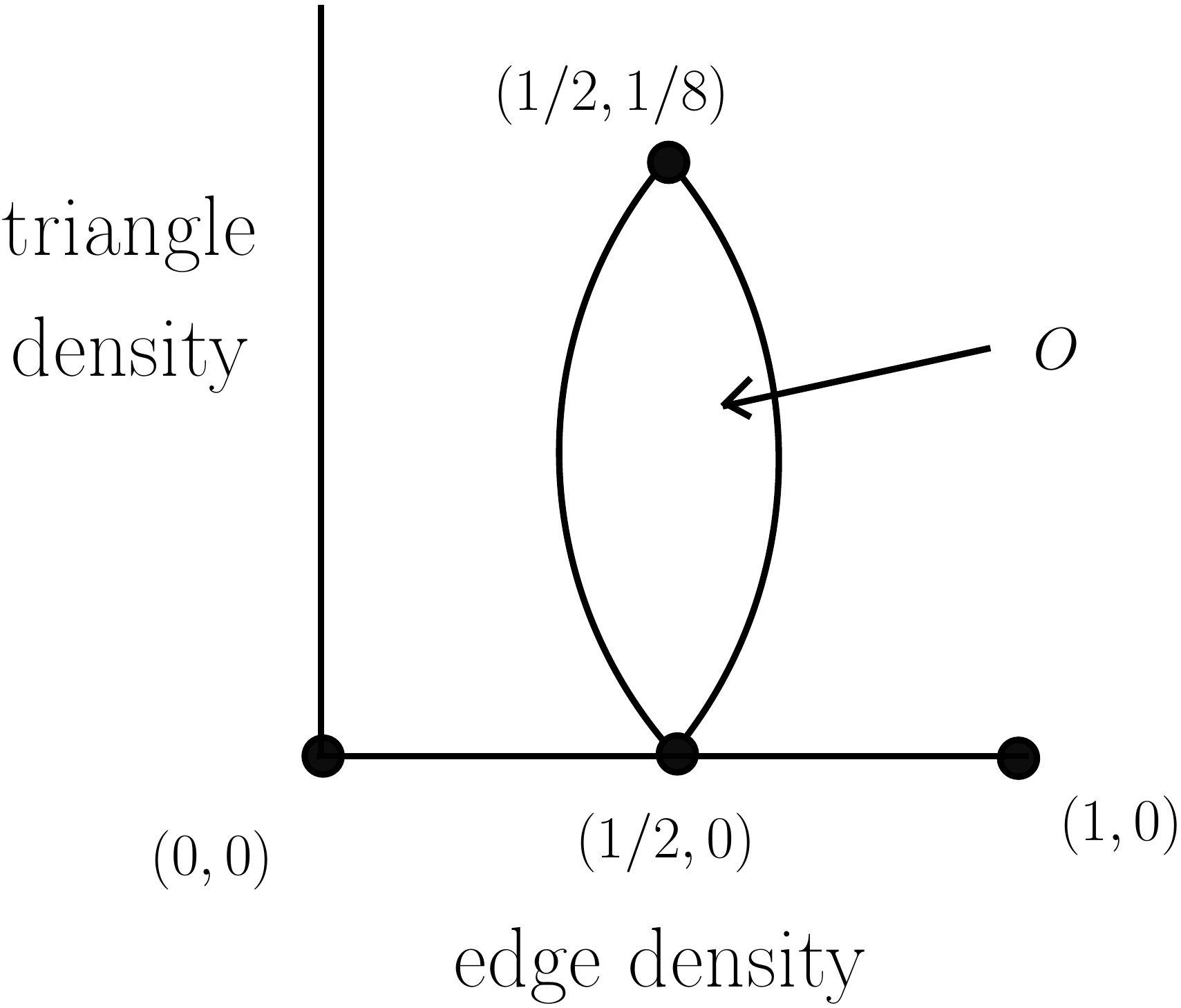}
\caption{The open set $\calO$ of Theorem \ref{main1}}
\label{FIG:openset}
\end{figure}

\begin{figure}[ht]
\centering
\includegraphics[angle=0,width=0.6\textwidth]{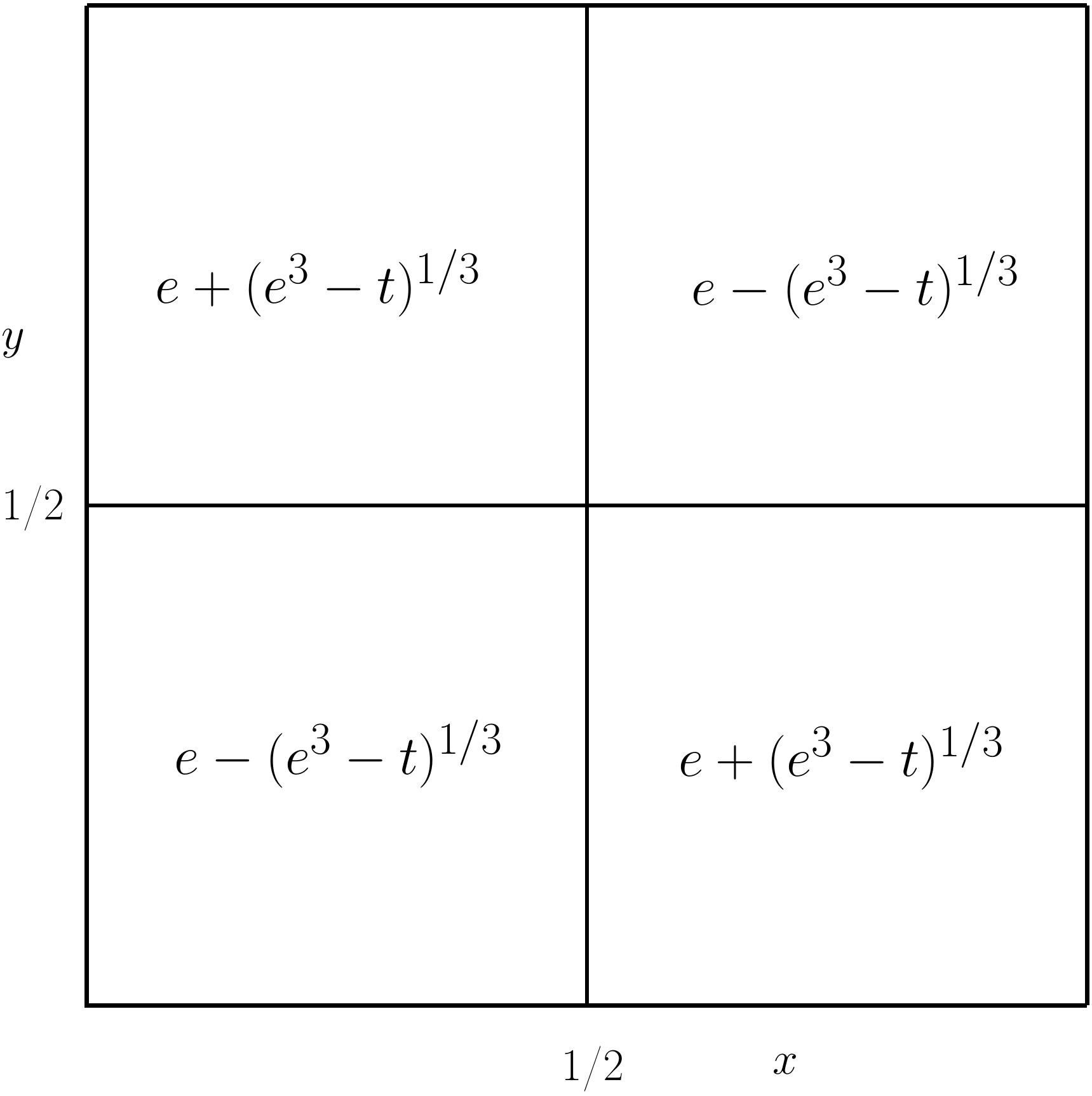}
\caption{The piecewise-constant optimal graphon of Theorem \ref{main1}}
\label{FIG:bipartite3c}
\end{figure}

Our second major result describes a region where the optimizing graphon is not symmetric bipodal.

\begin{theorem}\label{main3}
Let $e_0 = (3-\sqrt{3})/6 \approx 0.2113$. For any fixed edge density 
$e<e_0$ and any sufficiently small positive $\sigma$, the symmetric bipodal
graphon (\ref{symbipodal}) does not maximize $S$ among graphons with triangle 
density $t=e^3-\sigma^3$.
\end{theorem}

\subsection{Background}

This work is concerned with emergent features \cite{A} of large dense
simple graphs, features that are only meaningful in the limit
as the number $n$ of vertices goes to infinity. We will concentrate on
graphs with competing constraints, specifically the prescribed
densities $e$ and $t$, edges and triangles, respectively.  The
emergent behavior is {\em smoothness} as a function of $(e,t)$, of the
Boltzmann entropy $B(e,t)$ and of all subgraph densities.

The study of graphs with competing constraints is an old topic in
extremal combinatorics. For graphs, the range of achievable values of
the pair $(e,t)$, and the graphs that achieve them, was completed in
2012 by Purkurko and Razborov \cite{PR}: see Figure \ref{FIG:Razborov}
for a distorted view of the ``Razborov triangle''.

A graphon formalism was developed starting around 2006
\cite{BCL,BCLSV,LS1,LS2,LS3} to give a
useful meaning to the asymptotic limits of dense graphs.  A large
deviation principle (LDP) was added by Chatterjee and Varadhan in 2010
\cite{CV}. Using graphons we define a {\em
  phase} as a maximal connected open subset of the Razborov triangle
on which $B(e,t)$, and the density of every fixed subgraph (e.g., the
density of squares, pentagons, tetrahedra, $\ldots$) of a typical
graph, is an analytic
function of $(e,t)$. The system is said to have a {\em phase
  transition} wherever such quantities are not analytic. Usually this
occurs at the boundary of two or more phases, but sometimes there is
an analytic path from one side of a phase transition to another, as
discussed in Section \ref{sec:ordparam}. Note
that if the optimal graphon in the variational formula (\ref{variation}) is
not unique, the densities of some subgraphs are not even
well-defined, much less analytic; by definition, such points of non-uniqueness can
never lie within a phase. Even when uniqueness does hold, it can be
difficult to prove. On the other hand, where there is a
unique entropy-optimizing graphon $g_{e,t}$, all but exponentially few large graphs are close to
$g_{e,t}$; this facilitates the analysis of emergent features. (Such
uniqueness is known to fail in similar models; see Section 5 in
\cite{KRRS1}.)
\begin{figure}[ht]
\centering
\includegraphics[angle=0,width=0.7\textwidth]{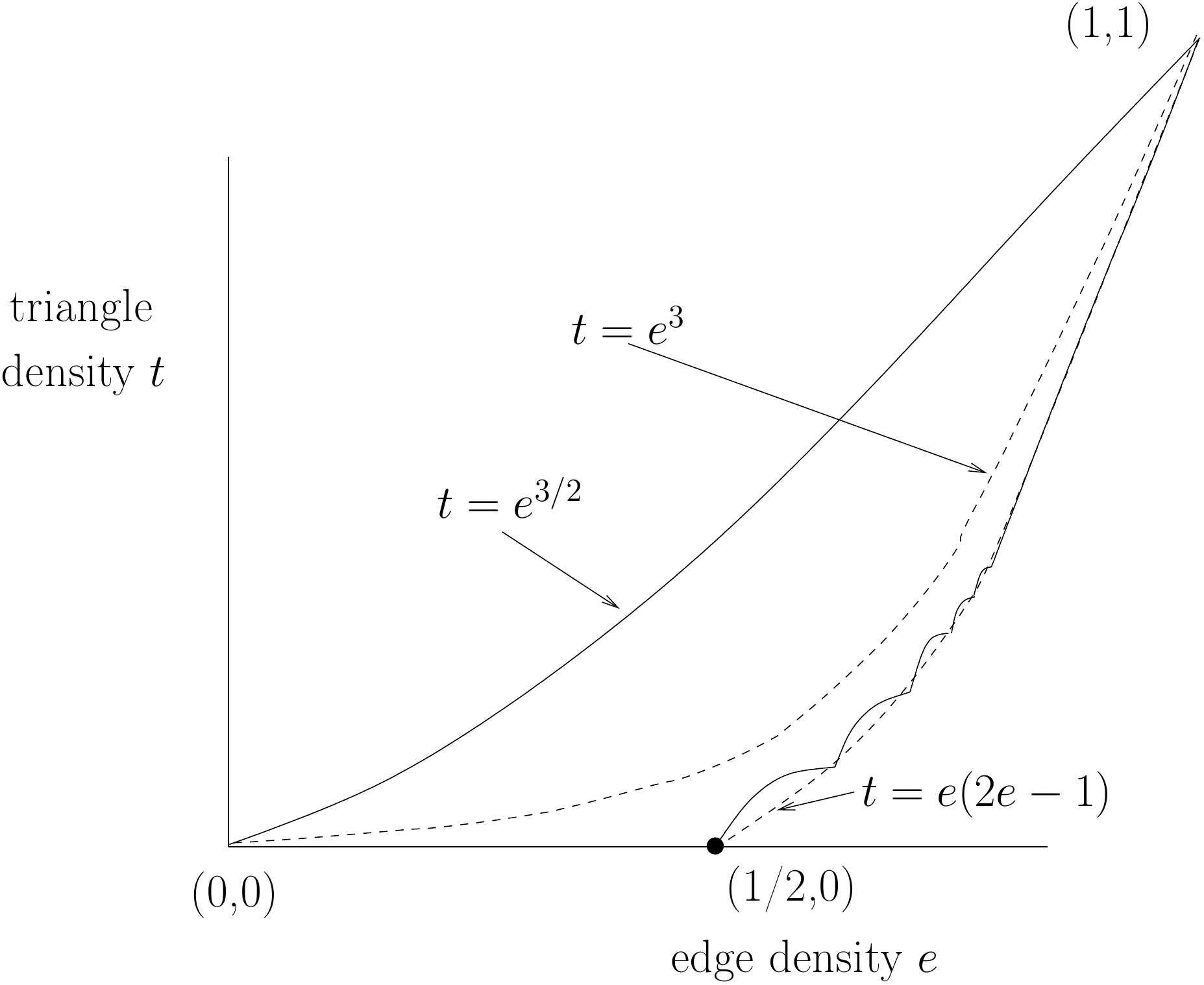}
\caption{The Razborov triangle, outlined in solid curves}
\label{FIG:Razborov}
\end{figure}

This paper is part of a project begun in 2013 \cite{RS1}, following
\cite{CD}, to study emergent analyticity of typical graphs with edge
and triangle constraints in the Razborov triangle. Using \cite{CV} we
derived \cite{RS1} the variational formula (\ref{variation}) relating
the Boltzmann entropy $B(e,t)$ to the graphon entropy $S(g)$. Let $G$
be a fixed subgraph, such as a square or pentagon or tetrahedron.
When $S(g)$ achieves the value $B(e,t)$ at a unique graphon, all but
exponentially few graphs have the same density $t_G$ of $G$, so we can
speak of $t_G$ being a function of $(e,t)$ and ask whether that
function is analytic.  (Our method was based on large deviations of
$G(n,p)$ \cite{CV}; for an approach based on $G(n,m)$ see \cite{DL}.)

In 2015 we established \cite{KRRS1} the existence of two open subsets
of the Razborov triangle, both with $t>e^3$, 
in which $B(e,t)$ and all subgraph densities were real-analytic in
$(e,t)$.
To prove this we proved that the constrained $S$-optimizing graphons
were unique, and determined that they have a 2-block (``bipodal'')
\cite{KRRS1} structure whose parameters were analytic in $(e,t)$.  
More recently, we proved
\cite{NRS2} a complementary result in the more difficult case of
undersaturated triangles ($t < e^3$) when $e > \frac12$.  This yielded
the satisfying result of a pair of open sets, $\calO_1,\calO_2$, on
which the Boltzmann entropy $B(e,t)$ and all subgraph densities are analytic, 
separated by the bounding curve $t=e^3$, on which the
entropy isn't even differentiable \cite{RS2}. See Figure
\ref{FIG:transition}.

\begin{figure}[ht]
\centering
\includegraphics[angle=0,width=0.6\textwidth]{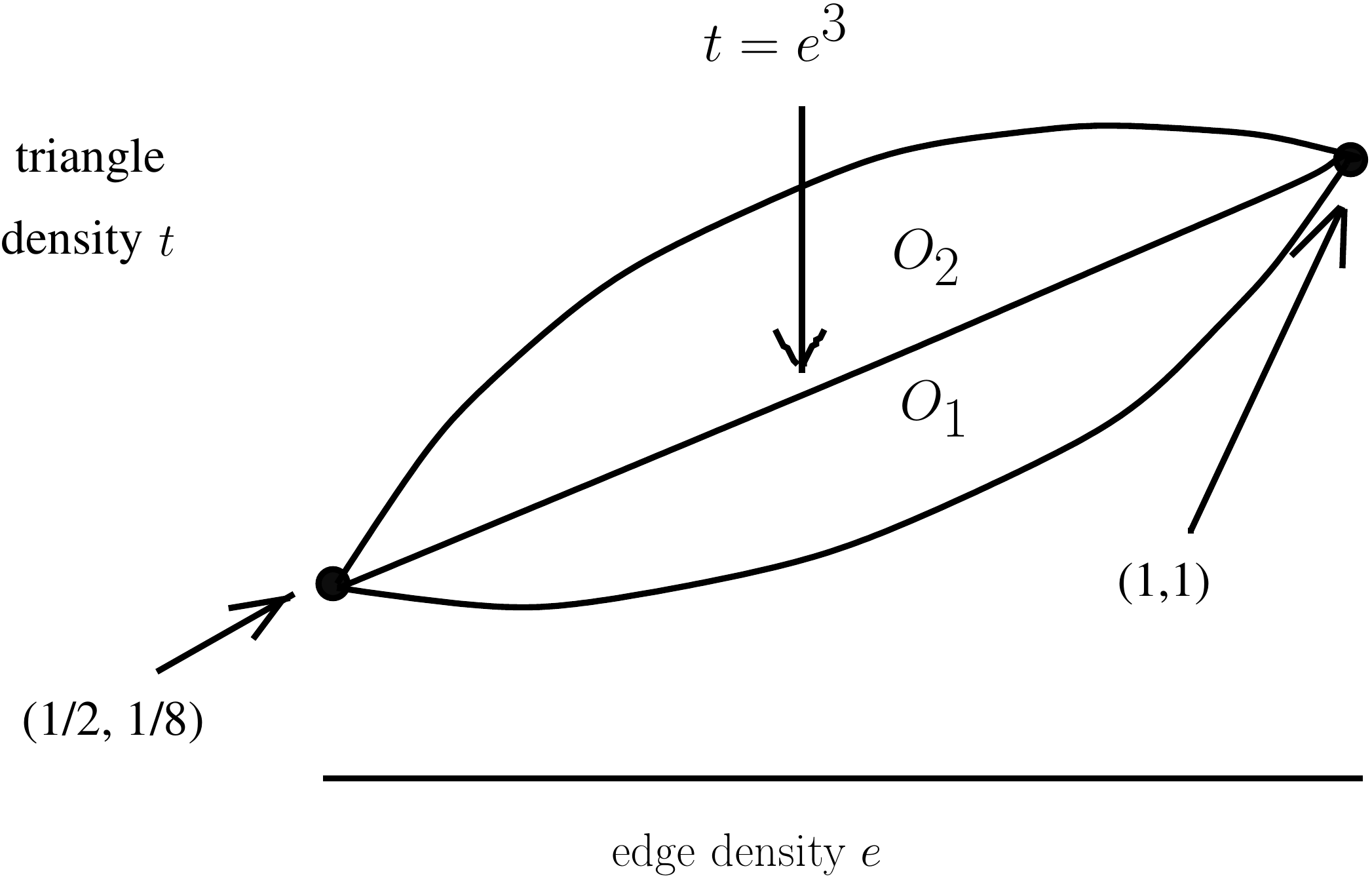}
\caption{The phase transition between $\calO_1$ and $\calO_2$ was proven in \cite{NRS2}.}
\label{FIG:transition}
\end{figure}

Put another way, in \cite{NRS2} we proved the existence of two phases
with $e>\frac12$, one just above the curve $t=e^3$ and one just below,
with a phase transition on the curve itself. As noted earlier, within
each phase there must be a unique $S$-optimizing graphon for each
$(e,t)$. This optimal graphon is called the ``state'' of the system.
As discussed above, models of dense graphs with sharp competing
constraints are a natural extension of extremal graph theory. We note
there have been parallel studies within other parts of extremal
combinatorics, for instance permutations and sphere packings; see
Section \ref{sec:Summ}.

One may reasonably ask how we know that the two open sets $\calO_1$
and $\calO_2$ actually belong to different phases; how can we rule out
the possibility that there is an analytic path between them, going
{\em around} the phase transition? Neither the $\calO_1$ phase nor the
$\calO_2$ phase seems to have any intrinsic property that clearly rules out an
analytic continuation between them.  Such an analytic
continuation does not actually exist, since we previously showed
\cite{RS2} that $B(e,t)$ cannot be differentiable at any point on
the curve $t=e^3$ for any $0<e<1$.
However that isn't a very satisfying explanation.

The ``symmetric bipodal'' phase whose existence we prove in this paper is different. The symmetry provides an {\em
intrinsic} difference between the new phase and the $\calO_1$ and $\calO_2$ phases. In Section \ref{sec:ordparam} 
we discuss the connection between our symmetry argument and the use of ``order
parameters'' in equilibrium statistical mechanics, and as a by-product
we clarify a problematic argument of Landau from the 1950s.

Another key difference between our new phase and the previously proven phases is that proof of 
the symmetric bipodal phase is not limited to a small neighborhood of the curve $t=e^3$. In \cite{KRRS1}, and 
again in \cite{NRS2}, we studied small perturbations of the Erd\H{o}s-R\'enyi graph $G(n,p)$ and attempted to get the
greatest possible change in triangle count for the smallest possible entropy cost. The results are closely related
to moderate deviation estimates \cite{NRS3}; depending on the sizes of $n$ and $e^3-t$, a finite graph with triangle 
density slightly less than $e^3$ can be viewed either as a typical $(e,t)$ graph or as a deviation of an Erd\H{o}s-R\'enyi
graph. When $e < 1/2$, moderate deviations estimates that apply when $n^{-1} \ll e^3-t \ll 1$ agree to leading order
with large deviations estimates.

This is reflected in the different method of proof. In Theorem \ref{main3}, $\sigma$ is not a small parameter. We 
can still do a power series expansion in $\sigma$, but we have to estimate all terms, not just the first few.

\section{Definitions and Notation}

\begin{figure}[ht]
\centering
\includegraphics[angle=0,width=0.4\textwidth]{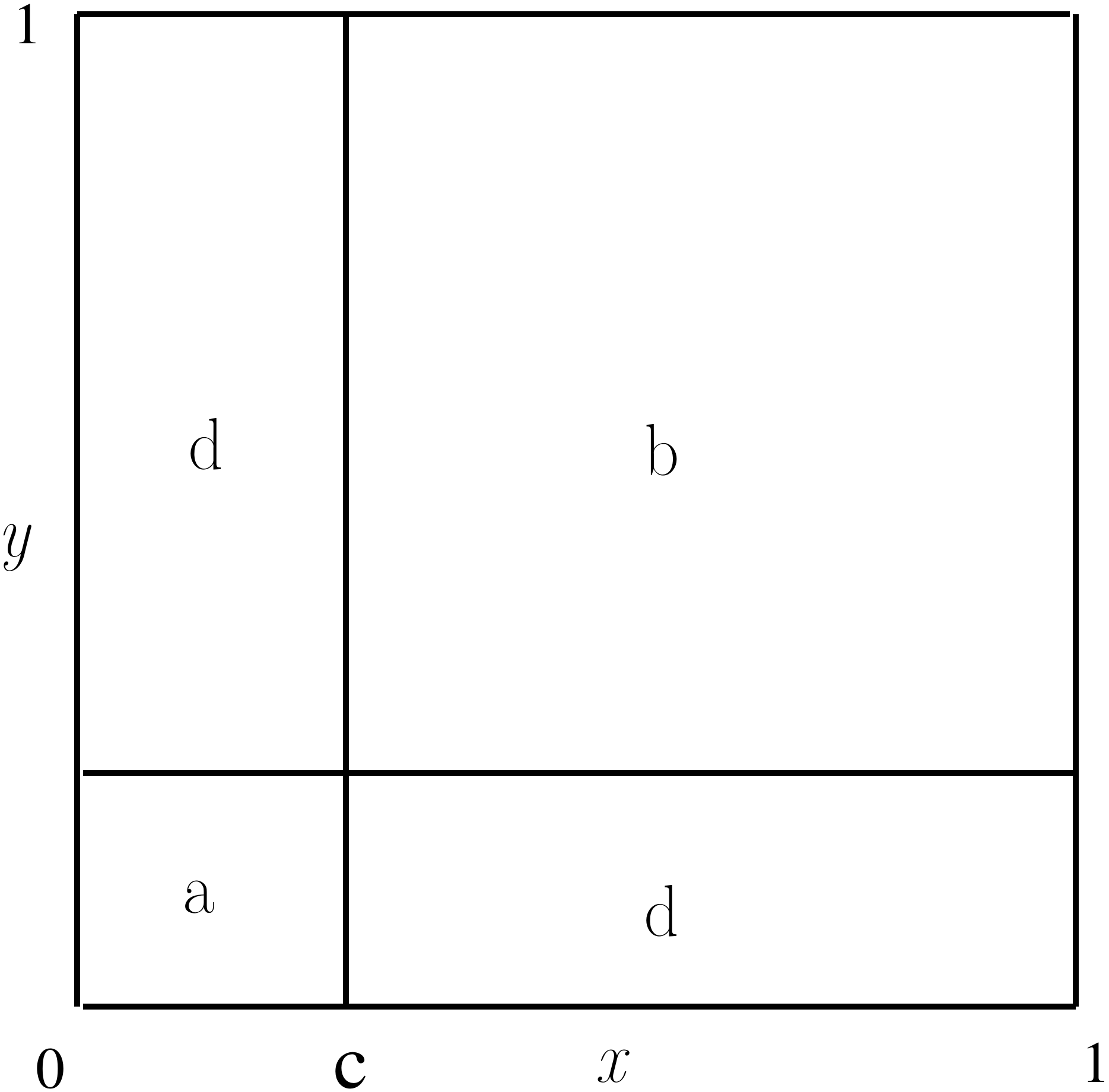}
\caption{The parameters of a bipodal graphon}
\label{FIG:bipodal}
\end{figure}

A graphon is said to be {\em bipodal} if it is equivalent to a graphon with the block structure shown in Figure 
\ref{FIG:bipodal}. 
It is {\em symmetric bipodal} if $c = 1/2$ and $a=b$. To avoid questions about graphons being equivalent under 
measure-preserving transformations of $[0,1]$, we restate the definition using arbitrary measurable subsets $I_1$ and 
$I_2$ of $[0,1]$, rather than intervals $[0,c]$ and $(c,1]$. A graphon is bipodal if there exist complementary
measurable subsets $I_1$ and $I_2$ such that $g(x,y)$ is constant on $I_1 \times I_1$, constant on $I_2 \times I_2$, 
and constant on $I_1 \times I_2 \cup I_2 \times I_1$. 
A graphon is symmetric bipodal if there exist complementary
subsets $I_1$ and $I_2$, each of measure 1/2, and a positive number $\sigma < e$, such that the graphon is 
\begin{equation}\label{symbip} g(x,y) = \begin{cases} e-\sigma & (x,y) \in I_1\times I_1 \cup I_2 \times I_2, \cr 
e+\sigma & (x,y) \in I_1\times I_2 \cup I_2 \times I_1, \end{cases} \end{equation}
The edge density,
triangle density and entropy of a symmetric bipodal graphon are
\begin{equation} \edens(g)=e, \qquad \tdens(g)=e^3-\sigma^3, \qquad S(g) = \frac12 (H(e+\sigma) + H(e-\sigma)).
\end{equation}
Of course, this only makes sense if $\sigma \le \min(e, 1-e)$.  
Another characterization of a symmetric bipodal graphon
is that it is a rank-1 perturbation of a constant graphon, with \[ g(x,y) = e -\sigma v_1(x) v_1(y), \]
where $\int_0^1 v_1(x) dx=0$ and $v_1(x)^2=1$ everywhere. 


It was previously known that the unique optimal graphon on the open line interval $e=1/2$, $t \in (0,1/8)$ 
was symmetric bipodal \cite{RS2}. Our main result, Theorem \ref{main1}, extends this
to an open set $\calO$ containing the line interval. It is convenient for our proofs to reformulate
Theorem \ref{main1} as
\begin{theorem}\label{main2}
For fixed $\sigma \in \left (0 ,\frac12\right )$ and for all
sufficiently small $\delta$ (of either sign), the unique $S$-maximizing
graphon with 
edge density $e=\frac12+\delta$ and triangle density $t=e^3-\sigma^3$ is symmetric bipodal. Furthermore, the size of the allowed
interval of $\delta$'s varies continuously with $\sigma$. 
\end{theorem}
\begin{corollary} The Boltzmann entropy $B(e,t)$ and the densities of all subgraphs
are real analytic functions of $(e,t)$ in the open set thus defined.\end{corollary}

We expect that the region where the optimal graphon is symmetric bipodal is not
limited to the small open set described in Theorems \ref{main1} and
\ref{main2}. There is considerable numerical evidence that this
region, called the A(2,0) phase in \cite{KRRS3}, is much bigger than
that. However, there are provable limits to its extent. Theorem \ref{main3} says that it does not 
extend to the curve $t=e^3$ when $e < e_0 \approx 0.2113$. Theorem 1 from \cite{NRS2}, which we restate here,
says that it does not extend to the curve $t=e^3$ when $e>1/2$. It is an open question whether the phase extends 
to the curve $t=e^3$ when $e_0 < e < 1/2$.

\begin{theorem}\label{thm:b11} (Theorem 1 from \cite{NRS2})
There is an open subset ${\calO}_1$ in the planar set of achievable parameters $(e,t)$,
whose upper boundary is the curve $t = e^3,\  1/2<e<1$, such that
at $(e,t)$ in ${\calO}_1$ there is a unique entropy-optimizing
graphon $g_{e,t}$. This graphon is bipodal and for fixed $(e,t)$, the values of 
$a,b,c,d$ can be approximated to arbitrary accuracy via an explicit iterative scheme. These parameters can also be expressed via asymptotic power series in $\sigma=(e^3-t)^{1/3}$ whose leading terms
are:
\begin{eqnarray}\label{approximation:b11-abcd}
a & = & 1-e - \sigma + O(\sigma^2) \cr 
b & = & e - \frac{\sigma^2}{2e-1} + O(\sigma^3) \cr 
c & = & \frac{\sigma}{2e-1} - \frac{2\sigma^2}{2e-1} + O(\sigma^3) \cr 
d & = & e + \sigma + \frac{\sigma^2}{eH'(e)} \left ( H'(e) - \left (e-\frac12 \right ) H''(e) \right ) + O(\sigma^3). \end{eqnarray}

\end{theorem}
\begin{corollary}\ $B(e,t)$ and the densities of all subgraphs
are real analytic in $(e,t)$ in the open set $\calO_1$.\end{corollary}
This corollary was proven in the last paragraph of the proof of
Theorem 1 in \cite{NRS2}, although not included in the statement of the theorem.

The bulk of this paper is devoted to proving Theorem \ref{main2},
which is tantamount to proving
Theorem \ref{main1}. 
To explain the steps, we need some more notation. 
We diagonalize $g(x,y) - e$, viewed as an integral operator, and write 
\begin{equation}\label{expand-g1} g(x,y) = e + \sum_{j=1}^\infty \lambda_j v_j(x) v_j(y), \end{equation}
where $|\lambda_1| \ge |\lambda_2| \ge \cdots$ and the functions $v_1, v_2, \cdots$ are orthonormal in $L^2([0,1])$. 
Let 
\begin{equation}\label{expand-g2} 
g_1(x,y) = \lambda_1 v_1(x) v_1(y), \qquad g_2(x,y) = \sum_{j=2}^\infty \lambda_j v_j(x) v_j(y). 
\end{equation}
Our goal is to show that $g_2 = 0$ and that $v_1(x)=\pm 1$, taking each value on a set of measure 1/2. We do this in stages:
\begin{itemize}
\item In Section \ref{sec:apriori}, we prove {\em a priori} entropy bounds on any graphon having the given values of 
$(e,t)$. We show that the symmetric bipodal graphon comes within $O(\delta^2)$ of saturating those bounds. This implies that 
any entropy-maximizing graphon must be $L^2$-close to a symmetric bipodal graphon. Specifically, 
$g_2$ must be $L^2$-small and $v_1$ must be $L^2$-close to the desired step function. 
\item In Section \ref{sec:pointwise} we show that $g_2$ is {\em pointwise} small and that $v_1(x)^2$ is pointwise close
to 1. More precisely, 
the $L^\infty$ norms of $g_2$ and $v_1^2-1$ must go to zero as $\delta \to 0$. 
\item In Section \ref{sec:cost-benefit} we expand the entropy $S(g)$ using a convergent Taylor series for $H(u)$ around $u=\frac12$. 
Using the fact that $g_2$ is pointwise small, we express the difference between $S(g)$ and the entropy of a symmetric bipodal graphon
as a quadratic function of the $L^2$ norm of $g_2$, the $L^2$-norm of $v_1(x)^2-1$ and the integral $\int_0^1 v_1(x) dx$, plus higher-order
corrections. We show that the quadratic function is negative-definite, implying that 
$g_2$ must be zero, $v_1(x)^2$ must be 1, and $\int_0^1 v_1(x) dx$ must be zero. In other words, our graphon must be symmetric bipodal. 
\item In Section \ref{sec:belowER} we turn our attention to Theorem \ref{main3}. We construct a family of tripodal graphons and we 
express the entropy of both this tripodal graphon and the symmetric bipodal graphon as power series in $\sigma$. When $e<e_0$, we can
choose the parameters of the tripodal graphon such that the tripodal graphon has more entropy at order $\sigma^2$ 
than the symmetric bipodal graphon. This does {\em not} prove that the optimal graphon is tripodal! However, it does proves that, 
for $\sigma$ sufficiently small, the symmetric bipodal graphon is not optimal.
\end{itemize}

We use big-O and little-o notation throughout. When we say that 
a certain quantity is $O(\delta^n)$, we mean that there exist positive
numbers $C$ and $\delta_0$ (which may depend on $\sigma$) 
such that our quantity is bounded by $C |\delta|^n$ whenever $|\delta|
< \delta_0$. When we say that a quantity is $o(\delta^n)$, we mean that
there exists a constant $\delta_0$ and function $f(\delta)$, going to zero
as $\delta \to 0$, such that the quantity is bounded by 
$f(\delta) |\delta^n|$ when $|\delta| < \delta_0$. 

\section{A priori estimates}\label{sec:apriori}

We begin with an upper bound on entropy. 
\begin{theorem}
If $g$ is a graphon with edge density $e=\frac12 + \delta$ and triangle density $e^3-\sigma^3$, 
with $\sigma>0$, then 
\begin{equation}\label{S-bound} S(g) \le H \left ( \frac12 + \sqrt{\delta^2 + \sigma^2}\right ). \end{equation}
\end{theorem}

\begin{proof}
Let $g$ be our arbitrary graphon, which we expand as in equations (\ref{expand-g1}) and (\ref{expand-g2}). For $i=1,2$, let 
$d_i(x) = \int_0^1 g_i(x,y) dy$. A direct computation of the triangle density gives 
\begin{eqnarray}\label{t-expand}
\tdens(g) & = & \iiint g(x,y) g(y,z) g(z,x) \, dx \, dy \, dz \cr 
 & = & e^3 + \sum_{j=1}^\infty \lambda_j^3 + 3e \int_0^1 (d_1(x)^2 + d_2(x)^2) dx \cr 
 \sum_j \lambda_j^3 & = & -(\sigma^3 + 3e \int_0^1 (d_1(x)^2 + d_2(x)^2) dx).
\end{eqnarray}
The squared $L^2$ norm of $g_1+g_2$ is
\[ \sum_j \lambda_j^2 \ge \left | \sum_j \lambda_j^3 \right |^{2/3} \ge \sigma^2, \]
with equality if and only if $\lambda_1=-\sigma$, $\lambda_2= \lambda_3 = \cdots = 0$, and $d_1(x)=0$ for all $x$.  

Next we maximize the entropy for a fixed $\| g_1+g_2 \|_2^2$. We use an absolutely convergent
power series for 
\begin{equation} H(u) = - (u \ln(u) + (1-u)\ln(1-u)), \end{equation}
namely 
\begin{equation}
H(u) = \sum_{n=0}^\infty \frac{H^{(n)}(\frac12)}{n!} \left (u-\frac12 \right )^n. 
\end{equation}
The terms with $n$ odd are identically zero, while the terms with $n$ 
nonzero and even are negative. As a result, 
\begin{equation}\label{S-expand}
S(g) = H\left (\frac12 \right ) + \sum_{k=1}^\infty \frac{H^{(2k)}(\frac12)}{(2k)!} \mu_{2k}, 
\end{equation}
where
\begin{equation}\label{moments} 
\mu_{2k} = \iint \left (\delta + g_1(x,y) + g_2(x,y) \right )^{2k} \, dx \, dy.
\end{equation}

The second moment depends only on the size of $g_1$ and $g_2$.
Since $\iint g_1(x,y)+g_2(x,y) \, dx \, dy =0$, 
there are no cross terms between $\delta$ and $g_1+g_2$, leaving us with 
\[ \mu_2 = \delta^2 + \| g_1 + g_2 \|_2^2 = \delta^2 
+ \sum_{j=1}^\infty \lambda_j^2. \]
Maximizing $S(g)$ is equivalent to minimizing all of the higher moments
$\mu_{2k}$ with $k>1$. This happens when 
$(\delta + g_1(x,y)+g_2(x,y) )^2$ is constant, equaling $\mu_2$. In that case,
$g$ is everywhere equal to $\frac12 \pm \sqrt{\mu_2}$ and 
\begin{equation} 
S(g) = H\left (\frac12 + \sqrt{\delta^2 + \| g_1+g_2 \|_2^2} \right ).
\end{equation}
Since $\| g_1+g_2 \|_2^2 \ge \sigma^2$, and since $H(u)$ is a 
decreasing function of $u$ for $u > 1/2$, we conclude that
$$S(g) \le H \left ( \frac12 + \sqrt{\delta^2 + \sigma^2} \right ). $$
\end{proof}

\begin{corollary}\label{cor:alpha_small} If $g$ is an entropy-maximizing graphon, then $\int_0^1 v_1(x) \, dx$ is $O(\delta)$, while 
$\| g_2 \|_2^2$, $\int_0^1 d_2(x)^2 \, dx$ and 
$\int_0^1 (v_1(x)^2-1)^2\, dx$ are $O(\delta^2)$. 
\end{corollary}

\begin{proof} The symmetric bipodal graphon has entropy 
\begin{eqnarray*}
\frac12 \left [H(e+\sigma) + H(e-\sigma) \right ] & = & \frac12 \left [ H\left (\frac12 + \delta + \sigma\right ) + 
H\left (\frac12 + \delta - \sigma \right ) \right ] \cr 
& = & \frac12 \left [ H\left (\frac12 + \sigma + \delta \right ) + 
H\left (\frac12 + \sigma -\delta \right ) \right ]  \cr 
& = & H \left ( \frac12 + \sigma\right ) + \frac{\delta^2}{2} H''\left ( \frac12 + \sigma \right ) + O(\delta^4).
\end{eqnarray*}
By contrast, the upper bound (\ref{S-bound}) is 
\[ H \left ( \frac12 + \sigma \right ) + \frac{\delta^2}{2\sigma} H'\left ( \frac12 + \sigma \right ) + O(\delta^4). \]

Since the symmetric bipodal graphon comes within $O(\delta^2)$ 
of achieving the upper bound (\ref{S-bound}), the entropy-maximizing 
graphon must also come within $O(\delta^2)$ of that bound. In particular,
$\|g_1 + g_2\|_2^2$ must be within $O(\delta^2)$ of $\sigma^2$ and the 
fourth moment $\mu_4$ can be no more than $O(\delta^2)$ greater
than $(\delta^2 + \sigma^2)^2=\sigma^4 + O(\delta^2)$. 

Now 
\[ \|g_1 + g_2 \|_2^2 = \sum_j \lambda_j^2 \ge \left | \sum_j \lambda_j^3
\right |^{2/3} = \left (\sigma^3 + 3e \int_0^1 d_1(x)^2 + d_2(x)^2 dx
\right )^{2/3}.\]
This can only be within $O(\delta^2)$ of $\sigma^2$ if 
$\int_0^1 d_1(x)^2 dx$ and $\int_0^1 d_2(x)^2 dx$ are both $O(\delta^2)$.
However, $\int_0^1 d_1(x)^2 dx = \lambda_1^2 \left (\int_0^1 v_1(x) dx
\right )^2$ and $\lambda_1 \approx -\sigma$, so $\int_0^1 v_1(x) dx$ must be $O(\delta)$. 

We now turn to $\|g_2\|_2^2 = \sum_{j=2}^\infty \lambda_j^2$. Since 
\[ \lambda_1^3 = -(\sigma^3 + 3e \int_0^1 d_1(x)^2+d_2(x)^2 \, dx  
- \sum_{j=2}^\infty \lambda_j^3), \]
and since $| \sum_{j=2}^\infty \lambda_j^3 | \le \|g_2\|_2^3$,
$\lambda_1 \le - \sigma + O(\|g_2\|_2^3)$. But then 
\[ \sum_{j=1}^\infty \lambda_j^2 \ge \sigma^2 + \|g_2\|_2^2 
+ O(\|g_2\|_2^3). \]
Since this must be within $O(\delta^2)$ of $\sigma^2$, we must have 
$\|g_2\|_2^2 = O(\delta^2)$. 

Finally, we consider the fourth moment $\mu_4$. The leading contribution
is 
\[ \lambda_1^4 \left ( \int_0^1 v_1(x)^4 dx \right )^2 
= \lambda_1^4 \left (1 + \int_0^1 (v_1(x)^2-1)^2 dx \right )^2. \] 
For this to be within $O(\delta^2)$ of $\sigma^4$,  
$\int_0^1 (v_1(x)^2-1)^2 dx$ must be $O(\delta^2)$. 
\end{proof} 

\section{Pointwise estimates}\label{sec:pointwise}
The upshot of Section \ref{sec:apriori} is that $g$ must be $L^2$-close to a symmetric bipodal graphon, with $\lambda_1$ being close to $-\sigma$, with the sum of the 
other $\lambda_j^2$ being small, and with $v_1(x)$ being close to 1 on a set of measure approximately 1/2 and close to $-1$ on a set of
measure approximately 1/2. In this section we upgrade those $L^2$ 
estimates into pointwise estimates:

\begin{proposition}\label{prop-g2-small}
If $g$ is an entropy-maximizing graphon, then $\|g_2(x,y)\|_\infty$ 
is $o(1)$. 
\end{proposition}

\begin{proposition}\label{prop-g1-small}
If $g$ is an entropy-maximizing graphon, then $\| v_1(x)^2-1\|_\infty$
is $o(1)$.
\end{proposition}

We prove Propositions \ref{prop-g2-small} and \ref{prop-g1-small}
with a series of lemmas. 
We begin by showing that $g_1$ and $g_2$ are pointwise bounded. 
\begin{lemma}\label{bound-g1} Let $g$ be an entropy-maximizing graphon.
For all $x, y \in [0,1]$, the following bounds apply: 
\[ |v_1(x)| \le |\lambda_1|^{-1}, \qquad 
|g_1(x,y)| \le |\lambda_1|^{-1}, \qquad |g_2(x,y)| \le 1 + 
|\lambda_1|^{-1}. \]
\end{lemma}

\begin{proof} We use the fact that 
\[ 0 \le e + g_1(x,y) + g_2(x,y) \le 1 \]
for all $(x,y)$. The only way for $g_2(x,y)$ be be big and positive 
(resp.~negative) is for $g_1(x,y)$ to be big and negative
(resp.~positive). This can only occur if $|v_1(x)|$ is large for some 
$x$. 

Suppose that there is a point $x_0$ with $v_1(x_0) > 1/|\lambda_1|$.
Let $I_+$ be the set of $x$ for which $v_1(x)>0.9 |\lambda_1 v_1(x_0)|^{-1}$ and let $I_-$ be the 
set of $x$ for which $v_1(x) < -0.9|\lambda_1 v_1(x_0)|^{-1}$. We already know that the set of points with $v_1(x)$ 
close to $\pm 1$ each have measure close to 1/2, since $\int_0^1 (v_1(x)^2-1) dx = O(\delta^2)$
and $\int_0^1 v_1(x) dx = O(\delta)$, so $I_+$ and $I_-$ also each have measure close to 1/2. 

 Since $\lambda_1$ is negative, $g_1(x,y)$ is less than or equal to
-0.9 for all $y \in I_+$ and is 
greater than or equal to 0.9 for all 
$y \in I_-$. Since $e$ is close to 1/2,  
$e+ g_1(x_0,y)$ is close to or less than -1.4 when $y \in I_+$ (and in particular
is less than -1.3) and is close to or greater than 1.4 (and in particular is 
greater than 1.3) when $y \in I_-$. As a result, $g_1(x_0,y)$ has magnitude at least 
0.3, and sign opposite to that of $g_2(x_0,y)$, for all $y \in I_+\cup I
I_-$. This means that $g_1(x_0,y) g_2(x_0,y) < -0.27$ for all $y \in I_+ \cup I_-$.

When $y \not \in I_- \cup I_+$, $|g_2(x_0,y)| \le 0.9 < 1$, so $|g_2(x_0,y)|<2$, so 
$g_1(x_0,y)g_2(x_0,y) < 2$. 
Since $I_+ \cup I_-$ is a set of measure $m \approx 1$, 
\[ \int_0^1 g_1(x_0,y) g_2(x_0,y) dy \le -0.27 m + 2(1-m) < 0.\]
However, 
\[ \int_0^1 g_1(x_0,y) g_2(x_0,y) \, dy = \sum_{j=2}^\infty 
\lambda_1 \lambda_j v_1(x_0) v_j(x_0) \int_0^1 v_1(y) v_j(y) \, dy = 0, 
\]
by the orthogonality of the functions $v_j(y)$ in $L^2([0,1])$.
This is a contradiction, so $x_0$ does not exist. 

The same argument, with signs reversed, rules
out the possibility that $v_1(x)$ is ever less than
$-1/|\lambda_1|$. Since $|v_1(x)|$ is bounded by $|\lambda_1|^{-1}$,
$|g_1(x,y)| = |\lambda_1 v_1(x) v_1(y)|$ is also bounded by
$|\lambda_1|^{-1}$. 
Finally, we have that 
\[ -e - g_1(x,y) \le g_2(x,y) \le 1-e - g_1(x,y). \]
Since $e$ and $1-e$ are both less than 1, this implies that 
$|g_2(x,y)| < |g_1(x,y)|+1 \le |\lambda_1|^{-1} + 1$.
\end{proof}

Lemma \ref{bound-g1} is stated in terms of $\lambda_1$, which of course depends on the graphon $g$. 
However, $\lambda_1 = -\sigma + o(1)$, so for small $\delta$ we can replace our bounds 
involving $\lambda_1$ with uniform bounds in terms of $\sigma$, at the cost of replacing the constant 1 with a 
slightly smaller number. For instance, 
\[ |v_1(x)| < 1.1 \sigma^{-1}, \qquad |g_1(x,y)| < 1.1\sigma^{-1},
\qquad |g_2(x,y)| \le 1 + 1.1 \sigma^{-1} \]
whenever $\delta$ is sufficiently small. In practice, we do not need the specific bounds of Lemma \ref{bound-g1}. 
All we really need is for $v_1(x)$, $g_1(x,y)$ and $g_2(x,y)$ to be bounded.

We next turn to showing that $g_2(x,y)$ is not only bounded but small. 
Since $\|g_2\|_2^2$ is $O(\delta^2)$, the set of points where 
$|g_2(x,y)|$ is not small (say, smaller than a fixed $\epsilon$) has
measure $O(\delta^2)$. We now establish a similar result for vertical strips. 

\begin{lemma}\label{g-mostly-small} 
Let $g$ be an entropy-maximizing graphon. 
For any $\epsilon>0$ and any $x \in [0,1]$, the set 
of $y$-values for which $|g_2(x,y)| > \epsilon$ has measure $o(1)$. 
\end{lemma}

\begin{proof}
Let $G(x,y) = \int_0^1 g(x,z) g(z,y) \, dz$. As as operator, this is 
the square of $g$. Expanding that square using $g(x,y) = e + g_1(x,y) + g_2(x,y)$, we $G_1$ be the portion of $G$ that comes from 
$e+g_1$, and let $G_2$ be the additional contributions that involve $g_2$. 
\begin{eqnarray}
G_1(x,y) & = &  \int_0^1 (e + g_1(x,z))(e + g_1(y,z)) dz \cr 
& = & e^2 + \lambda_1^2 v_1(x)v_1(y) + e \left (
\int_0^1 v_1(z) dz \right ) (v_1(x) + v_1(y)). \cr 
& = & e^2 + \lambda_1 g_1(x,y) + O(\delta), \cr 
& = & \lambda_1(e+g_1(x,y)) + e^2-\lambda_1 e + O(\delta),
\end{eqnarray}
since $\int_0^1 v_1(z)\, dz = O(\delta)$ and $v_1(x)$ and 
$v_1(y)$ are bounded. 

We next turn to $G_2$. Since $\int_0^1 v_1(z) g_2(y,z) \, dz = 0$, there is no contribution from
the product of $g_1$ and $g_2$. We only have $eg_2$ and $g_2^2$ terms, specifically
\[ G_2(x,y) = e (d_2(x) + d_2(y)) + \int_0^1 g_2(x,z) g_2(y,z) dz. \]
The function $d_2(y)$ has small $L^2$-norm, and so must be $o(1)$ except
on a set of small measure. (Since $x$ is fixed, we cannot similarly argue
that $d_2(x)$ is small.) Finally,
since $g_2$ is bounded and has small $L^2$ norm, 
the integral $\int_0^1 g_2(x,z) g_2(y,z) dz$ for fixed $x$ is small except for a set of $y$'s that has small measure. 
The result is an estimate 
\[ G_2(x,y) = e d_2(x) + o(1) \]
that is true for $y$ in the complement of a set of measure $o(1)$, where that small set may depend on $x$. 

Combining this with our estimate of $G_1$, we have 
\begin{equation}
G(x,y) = (e^2 - \lambda_1 e) + \lambda_1 g_1(x,y) + e d_2(x) + o(1)
\end{equation}
for all but a small set of $y$'s. 

Since $g$ is assumed to maximize entropy subject to constraints on 
$\edens(g)$ and $\tdens(g)$, the functional derivative of $S(g)$ must 
be a linear combination of the functional derivatives of $\edens(g)$ 
and $\tdens(g)$. This yields the pointwise equations 
\begin{equation}\label{EL-general} H'(g(x,y)) = \Lambda_e + \Lambda_t G(x,y), \end{equation}
where $\Lambda_e$ and $\Lambda_t$ are Lagrange multipliers. 

For all but a small set of $y$'s, equation (\ref{EL-general}) takes the form 
\begin{equation} \label{EL-g2} H'(e+g_1(x,y) + g_2(x,y)) = \mu + \nu (e + g_1(x,y)) 
+ \rho d_2(x) + o(1), \end{equation}
where 
\[ \mu = \Lambda_e + \Lambda_t (e^2-\lambda_1 e), \qquad \nu = \lambda_1 \Lambda_t, \qquad \rho = e \Lambda_t. \]

For most values of $(x,y)$, $e+g_1(x,y)$ is close to $e \pm \sigma$ and 
$g_2(x,y)$ and $d_2(x)$ are small. 
This fixes $\Lambda_e$ and $\Lambda_t$, and therefore $\mu$, $\nu$, and $\rho$, 
to within a small error. 
Since $e = \frac12 + \delta$ and $H'(\frac12 - \sigma) = - H'(\frac12 + \sigma)$, we obtain
\begin{eqnarray}
H'\left(\frac12 + \sigma \right ) & = & \Lambda_e + \Lambda_t \left (\frac14-\sigma^2 \right ) + o(1) \cr \cr 
-H'\left (\frac12 + \sigma \right ) & = & \Lambda_e + \Lambda_t \left (\frac14 + \sigma^2 \right ) + o(1),  
\end{eqnarray}
 with solution 
 \begin{eqnarray}
 \Lambda_e \approx -\frac{1}{4\sigma^2} H'\left (\frac12 + \sigma \right ), & 
\displaystyle{\Lambda_t \approx \frac{1}{\sigma^2} H'\left (\frac12 + \sigma \right ),} & \cr \cr
 \mu \approx -\frac{1}{2\sigma} H'\left (\frac12 + \sigma \right ), & 
\displaystyle{\nu \approx \frac{1}{\sigma} H'\left (\frac12 + \sigma \right ),} & 
\rho \approx -\frac{1}{2\sigma^2} H'\left (\frac12 + \sigma \right ),
\end{eqnarray}
where all of the approximations are ``$+ o(1)$'' as $\delta \to 0$. 

From the explicit form of $H'(g) = \ln(1-g)-\ln(g)$, we see that there 
are only three roots to the equation $H'(g) = \mu + \nu g$,
which are located near $g=e\pm \sigma$ and $g=e$. Our immediate goal
is to show that 
$v_1(x)$ only takes values close to 0 and $\pm 1$, which implies that 
$g_1(x,y)$ is only close to $e$ and $e \pm \sigma$. 

Since $\int v_1(x) dx$ and $\int (v_1(x)^2-1)^2 dx$ are small, 
the function $v_1$ must be
close to 1 on an interval (call it $I_1$) of measure close to 1/2, 
must be close to $-1$ on an interval $I_2$ of measure close to $1/2$, may
be close to 0 on a third interval $I_3$ of small measure, and may take on
other values on a fourth interval $I_4$ of small measure. 

Let $x$ be an arbitrary point in $[0,1]$, and let $y_1$ and $y_2$
be generic points in $I_1$ and $I_2$. Equation (\ref{EL-g2}) then determines
$a=g_2(x,y_1)$ and $b=g_2(x,y_2)$ in terms of $d_2(x)$. What's more, $g_2$ takes values close
to $a$ on all of $\{x\} \times I_1$ (excepting those values of $y$ where
equation (\ref{EL-g2}) doesn't apply), and takes values close to $b$ on all of 
$\{x\} \times I_2$. We then compute
\[ \int_0^1 g_2(x,y) v_1(y) dy \approx (a-b)/2. \]
However, this integral must be zero, since $v_1(y)$ is orthogonal to all of 
the functions $v_i(y)$ that make up $g_2(x,y)$. We conclude that $a \approx b$.

If $b \approx a$, then $d_2(x) = \int_0^1 g_2(x,y) dy$ also equals $a$ and our
equations at $(x,y_2)$ and $(x,y_1)$ become 
\begin{eqnarray}\label{H-prime-fun}
H'\left(\frac12 + \sigma v_1(x) + a \right ) & \approx & \mu + \nu \left (\frac12 +\sigma v_1(x)\right ) + \rho a \cr 
H'\left (\frac12 - \sigma v_1(x) + a \right ) & \approx & \mu + \nu \left (\frac12 - \sigma v_1(x) \right ) + \rho a,
\end{eqnarray}
Adding these equations, and using the fact that $H'(\frac12 - \sigma v_1(x) + a) = - H'(\frac12 + \sigma v_1 -a)$, 
we get 
\begin{equation}\label{no-a-value} H'\left (\frac12 + \sigma v_1(x) + a \right ) - H'\left (\frac12 + \sigma v_1(x) - a \right ) \approx  
2 \mu + \nu + 2 \rho a \approx - \frac{a}{\sigma^2} H'\left (\frac12 + \sigma\right ). \end{equation}

By the mean value theorem, the left hand side of equation (\ref{no-a-value}) is $2a H''(u)$ for some $u$
between $\frac12 + \sigma v_1(x)-a$ and $\frac12 + \sigma v_1(x) + a$. Regardless of the value of $u$, this is 
a negative multiple of $a$. However, the right hand side is a positive multiple of $a$, since $H'(\frac12 + \sigma)<0$.
Since a negative multiple of $a$ equals a positive multiple, $a$ must be (approximately) zero. 

In particular, $g_2(x,y) \approx 0$ for all $y$ such that equation (\ref{EL-g2}) applies. That is, for fixed $x$, 
$g_2(x,y)$ is close to zero except on a set of $y$'s of small measure. 
\end{proof}

\begin{proof}[Proof of Proposition \ref{prop-g2-small}]
By Lemma \ref{bound-g1}, $g_2(x,y)$ is bounded. By Lemma
\ref{g-mostly-small}, for each $x$, $g_2(x,y)$  
is small for all but a small set of $y$'s. Combining these results,
we see that the degree function $d_2(x)$ is everywhere small, 
as is the integral $\int_0^1 g_2(x,z) g_2(y,z)\, dz$. 
That is, $G_2(x,y)$ is pointwise small. But then equation (\ref{EL-g2}) applies everywhere, so $g_2(x,y)$ is small everywhere. 
\end{proof}

\begin{proof}[Proof of Proposition \ref{prop-g1-small}]
In the notation of the proof of Lemma \ref{g-mostly-small}, we must show that the intervals $I_3$ and $I_4$ are 
empty, implying that $v_1(x)$ is everywhere close to $\pm 1$. 

Since $G_2(x,y)$ is small for all $(x,y)$, we must have 
\[ H'\left (\frac12 + \sigma v_1(x) \right ) = \mu
+ \nu \left (\frac12 + \sigma v_1(x) \right ) \] 
for all $y \in I_2$. However, the only solutions to this equation are (approximately) $\sigma v_1(x) = \pm \sigma$ or 0,
implying that $v_1(x) \approx \pm 1$ or 0. In other words, $x \in I_1 \cup I_2 \cup I_3$ and $I_4$ is empty. 

Showing that $I_3$ is empty requires a completely different argument, since equation (\ref{EL-general}) is indeed satisfied when $x \in I_3$.
However, equation (\ref{EL-general}) only defines stationary points with respect to pointwise small changes in 
$g(x,y)$. We also have to consider infinitesimal changes in the boundary between $I_1$, $I_2$ and $I_3$. 

So suppose that we increase the size of $I_1$ by an amount $\epsilon$ at the expense of $I_3$. That is, we change the 
value of $v_1(x)$ from near 0 to near 1 on a set of small measure $\epsilon$. To first order in $\epsilon$, the change in 
entropy is $2\epsilon (H(\frac12 + \sigma) - H(\frac12))$, since we are changing $g(x,y)$ from near zero to near 
$\frac12 \pm \sigma$ on a set of measure $2 \epsilon + O(\epsilon^2)$, and since $H(\frac12-\sigma)= H(\frac12 + \sigma)$. 
The change in edge density is $-2\epsilon \lambda_1 \int_0^1 v_1(y) dy = O(\delta \epsilon)$. 
To leading order, the change in the triangle density 
is $3\lambda_1^3 \epsilon$, since the $\lambda_1^3$ contribution to $\tdens(g)$ is actually
$\lambda_1^3 (\int_0^1 v_1(x)^2 dx)^3$, which changes from 
$\lambda_1^3$ to $\lambda_1^3 (1+\epsilon)^3 \approx \lambda_1^3 (1+3\epsilon)$. 

The variational equations $dS = \Lambda_e d\edens 
+ \Lambda_t d\tdens/3$ then become 
\begin{equation} \label{vary-v} 
2 \left [H \left (\frac12 + \sigma\right ) - H\left (\frac12 \right ) \right ] = \lambda_1^3 \Lambda_t = \sigma H'\left (\frac12 + \sigma \right ). \end{equation}
We expand both sides  of equation (\ref{vary-v}) as power series in 
$\sigma$. The left-hand side is 
\[ 2 \sum_{k=1}^\infty \frac{H^{(2k)}(1/2)}{(2k)!} \sigma^{2k}. \]
The right-hand side is 
\[ \sum_{k=1}^\infty \frac{H^{(2k)}(1/2)}{(2k-1)!} \sigma^{2k}. \]
The coefficients agree when $k=1$, but are strictly greater for the right-hand side when $k>1$. Since all terms are 
strictly negative (insofar as all even derivatives of $H$ are negative-definite), the right-hand side strictly smaller than the left-hand side.

Since varying the size of $I_3$ does not satisfy the variational equation (\ref{vary-v}), 
we cannot be in the interior of our parameter space. Rather, the measure of $I_3$ must be zero. 
\end{proof}

\section{Cost-benefit analysis}\label{sec:cost-benefit}

In Sections \ref{sec:apriori} and \ref{sec:pointwise} we showed that $g_2(x,y)$ is pointwise small, as is $v_1(x)^2-1$. 
In this section we show that they are 
zero, completing the proof of Theorem \ref{main1}. 
The key measures of how far they are from being zero are 
\begin{equation} \alpha^2 = \sum_{i=2}^\infty \lambda_i^2, \qquad 
\beta^2 = \int_0^1 (v_1(x)^2-1)^2 \, dx, \quad \hbox{ and } \quad \gamma = \int_0^1 v_1(x)\, dx.
\end{equation}
By Corollary \ref{cor:alpha_small}, $\alpha$, $\beta$ and $\gamma$ are all $O(\delta)$. 
The symmetric bipodal graphon is characterized by $\alpha=\beta=\gamma=0$. 

We use the expansion (\ref{S-expand}) and compare the moments $\mu_{2k}$ to those of the symmetric bipodal graphon.
We will estimate costs (terms that increase $\mu_{2k}$) and benefits (terms that decrease $\mu_{2k}$). 
We will show that having $\alpha$ or $\beta$ or $\gamma$
nonzero comes with costs
that go as $\alpha^2$, $\beta^2$, and $\gamma^2$, while the benefits 
are $o(\alpha^2+\beta^2+\gamma^2)$. When $\delta$ is sufficiently small, the costs exceed the benefits, so the 
symmetric bipodal graphon has more entropy than any graphon that isn't symmetric bipodal. 

We first establish the costs. Our triangle density is 
\[ e^3-\sigma^3 = t = e^3 + 3e \int_0^1 \left (d_1(x)^2 + d_2(x)^2\right ) \, dx + \lambda_1^3 + \sum_{j=2}^\infty \lambda_j^3. \] 
Now 
\[ \int_0^1 d_1^2(x) dx = \lambda_1^2 \left (\int_0^1 v_1(x) \, dx \right )^2 = \lambda_1^2 \gamma^2, \]
while $\int_0^1 d_2(x)^2 \, dx > 0$ and 
\[ \left |\sum_{j=2}^\infty \lambda_j^3 \right | \le \left ( \sum_{j=2}^\infty \lambda_j^2\right )^{3/2} = \alpha^3.\]
This implies that 
\[ \lambda_1^3 \le -\sigma^3 - 3 e \sigma^2 \gamma^2 + O(\alpha^3), \]
so
\[ \lambda_1^2 \ge \sigma^2 + 2e\sigma \gamma^2 + O((\alpha, \beta, \gamma)^3)  \]
and 
\[ \mu_2 = \lambda_1^2 + \alpha^2 \ge \sigma^2 + \alpha^2 + 2e \sigma \gamma^2 + o(\alpha^2+\beta^2+\gamma^2).\]
That is, there are $\alpha^2$ and $\gamma^2$ costs associated with $\mu_2$. 

We next look at $\mu_4$. This contains a term  
\[ \iint g_1(x,y)^4 \, dx \, dy = \lambda_1^4 \left ( \int_0^1 v_1(x)^4 dx \right )^2 = \lambda_1^4 (1+\beta^2)^2
\ge \sigma^4 (1 + 2\beta^2). \]
That is, there is a cost proportional to $\beta^2$. Having established costs proportional to $\alpha^2$, $\beta^2$, 
and $\gamma^2$, we just have to show that the benefits of having $\alpha, \beta, \gamma$ nonzero are smaller. 

We are looking at even moments 
\[ \mu_{2k} = \iint (\delta + g_1 + g_2)^{2k} \, dx \, dy. \]
We expand out the power, getting terms proportional to a power of $\delta$ times a power of $g_1$ times a power of $g_2$. 
We make repeated use of the following trick: 
\[ u^{2m} = 1+ u^{2m} - 1 = 1 + (u^2-1) p_m(u), \]
where
\[ p_m(u) = 1 + u^2 + u^4 + \cdots + u^{2m-1}. \]
This means that 
\begin{eqnarray}\label{k-even}  g_1(x,y)^{2m} & = &   \lambda_1^{2m} \left [ 1 + \left (v_1(x)^2-1 \right ) p_m(v_1(x))\right ]  \left [ 1 + \left (v_1(y)^2 -1 \right )p_m(v_1(y))\right ] \cr 
& = & \lambda_1^{2m} \Big [ 1 + \left (v_1(x)^2-1 \right ) p_m(v_1(x)) + \left (v_1(y)^2-1 \right ) p_m(v_1(y)) \cr 
&& \qquad + \left(v_1(x)^2-1 \right ) \left (v_1(y)^2-1 \right ) p_m(v_1(x)) p_m(v_1(y)) \Big ]
\end{eqnarray} 
and 
\begin{eqnarray}\label{k-odd}  g_1(x,y)^{2m+1} & = & \lambda_1^{2m+1} \Big [
v_1(y) \left (v_1(x) + 
\left (v_1(x)^3-v_1(x) \right ) p_m(v_1(x)) \right ) \cr 
&& \qquad + v_1(x) \left (v_1(y)^3-v_1(y) \right ) p_m(v_1(y)) \cr 
&& \qquad + \left (v_1(x)^2-1 \right ) \left (v_1(y)^2-1 \right ) v_1(x) v_1(y) p_m(v_1(x)) p_m(v_1(y)) \Big ].
\end{eqnarray} 

We divide the terms obtained by expanding $\iint (\delta + g_1 + g_2)^{2k}$ into several classes. 
\begin{enumerate}
\item Terms with three or more powers of $g_2$. Since $g_2$ is pointwise small and $g_1$ is bounded, these are bounded 
by small multiples of $\iint g_2^2 = \alpha^2$. In other words, they are $o(\alpha^2)$. 
\item Terms with two powers of $g_2$, an even number of powers of $\delta$ and an even number of powers of $g_1$. 
These are manifestly positive and represent costs, not benefits. Aside from the $\iint \delta^0g_1^0 g_2^2 = \alpha^2$ 
contribution to $\mu_2$ that we already considered, we do not keep track of these. 
\item Terms with two powers of $g_2$, an odd number of powers of $\delta$ and an odd number of powers of $g_1$. The integrand is a positive power of $\delta$ times a bounded quantity times $g_2(x,y)^2$, making the integral $O(\delta
\alpha^2)$.
\item Terms with one power of $g_2$, an odd power of $\delta$ and an even power of $g_1$.  We expand these using 
equation (\ref{k-even}) and consider each piece separately. First, we compute 
\[ \iint g_2(x,y) dx \, dy = - \iint g_1(x,y) dx \, dy = -\lambda_1 \left ( \int_0^1 v_1(x) dx \right )^2 
= - \lambda_1 \gamma^2.\]
Next, the $L^2$ norm of $v_1(x)^2-1$ is $\beta$ and the $L^2$ norm of $g_2$ is $\alpha$,
so 
\[ \iint (v_1(x)^2-1) p_m(v_1(x))  g_2 \, dx \, dy
= O(\alpha \beta). \]
The $(v_1(y)^2-1)$ piece is similar, while the $(v_1(x)^2-1)(v_1(y)^2-1)$ piece is $O(\alpha \beta^2)$. 
Thus $\iint g_1^{odd} \, g_2 \, dx \, dy = O(\alpha \beta)+ O(\gamma^2)$, so 
$\iint \delta^{odd}\, g_1^{odd}\, g_2 \, dx \, dy = O(\delta \alpha \beta) + O(\delta \gamma^2) 
= o(\alpha^2+\beta^2+\gamma^2)$. 
\item Terms with one power of $g_2$, an even power of $\delta$ and an odd power of $g_1$. We expand these using 
equation (\ref{k-odd}), noting that the first line in (\ref{k-odd}) has a single factor of $v_1(y)$ and the second line has a single factor of $v_1(x)$. However, 
\[ \int_0^1 v_1(y) g_2(x,y) \, dy= \int_0^1 v_1(x) g_2(x,y) \, dx = 0, \]
where $x$ is arbitrary in the first integral and $y$ is arbitrary in the second. Thus the first two lines contribute nothing and 
$ \iint g_1(x,y)^{2m+1} g_2(x,y) \, dx \, dy$ is equal to  
\[ \lambda_1^{2m+1} \iint \Big( v_1(x)^2-1 \Big) \Big (v_1(y)^2-1 \Big ) \Big ( v_1(x) v_1(y) p_m(v_1(x)) p_m(v_1(y) \Big ) g_2(x,y)\, dx \, dy.
\]
The factor $\Big( v_1(x)^2-1 \Big) \Big (v_1(y)^2-1 \Big )$ has $L^2$ norm $\beta^2$, the factor $g_2(x,y)$ has $L^2$ norm $\alpha$, and the factor 
$ v_1(x) v_1(y) p_m(v_1(x)) v_1(y) p_m(v_1(y))$ is bounded, so the 
integral is $O(\alpha\beta^2) = o(\alpha^2+\beta^2+\gamma^2)$. 
\item Terms with no powers of $g_2$, an even number of powers of $\delta$ and an even number of powers of $g_1$. 
These are all positive and are at least as big as the corresponding terms for the symmetric bipodal graphon. We have 
already taken into account the costs associated with $\mu_2$ and $\mu_4$. There are additional costs associated 
with higher moments, but they are not needed for this proof. 
\item Finally, there are terms with no powers of $g_2$, an odd number of powers of $\delta$ and an odd number of 
powers of $g_1$. Note that 
\[ \iint g_1(x,y)^{2m+1} \, dx \, dy = \lambda_1^{2m+1} \nu_{2m+1}^2,\]
where
\begin{eqnarray*}
\nu_{2m+1} & = & \int_0^1 v_1(x)^{2m+1} dx \cr 
& = & \int_0^1 v_1(x) \, dx+ \int_0^1 \left (v_1(x)^2-1 \right ) v_1(x) p_m(v_1(x))\, dx \cr 
& = & \gamma + O(\beta),
\end{eqnarray*}
since $v_1(x)^2-1$ has $L^2$-norm $\beta$ and $v_1(x) p_m(v_1(x))$ is bounded. 
Squaring $\nu_{2m+1}$ and
multiplying by an odd power of $\delta$ gives $O(\delta\gamma^2) + O(\delta \beta^2) + O(\delta \beta \gamma)$,
which is $o(\alpha^2+\beta^2+\gamma^2)$. 
\end{enumerate}

Putting everything together, we have identified costs proportional to $\alpha^2$, $\beta^2$ and $\gamma^2$. Other costs
only add to that total, so the total cost is at least a constant times $\alpha^2+\beta^2+\gamma^2$. All of 
the potential benefits are smaller, either involving three or more powers of $\alpha$, $\beta$ and $\gamma$, 
or $\delta$ times a quadratic function of $\alpha$, $\beta$ and $\gamma$, or the sup norm of $g_2$ times $\alpha^2$. 
When $\delta$ is sufficiently small, the costs outweigh the benefits, so the optimal graphon is symmetric bipodal. 

\section{The extent of the symmetric bipodal phase}\label{sec:belowER}

We have proven that the entropy-maximizing graphon is unique and symmetric 
bipodal on a region containing the interval $e=1/2$, $0 < t < 1/8$. It 
is natural to ask how far this symmetric bipodal phase extends. There
is considerable evidence that this phase contains much of the region 
$e \le 1/2$, $t<e^3$. 
\begin{itemize}
\item The unique entropy-maximizing graphon when $e<1/2$ and $t=0$ is known to be symmetric bipodal, with $\sigma = e$ \cite{RS1}. 
The entropy is $\frac12 H(2e)$. 
\item When $e<1/2$ and $t<e^3$, the symmetric bipodal graphon has entropy
strictly higher than any other graphon of the form 
$g(x,y) = e + \lambda_1 v_1(x) v_1(y)$. This is Proposition
\ref{rank-1}, proven below. 
\item When $e<1/2$ and $t$ sufficiently close to but below $e^3$, the symmetric bipodal graphon has 
entropy strictly higher than any other bipodal graphon. This is 
Proposition \ref{no-bipodal}, proven below.
\item Numerical investigations of the region $e<1/2$, $t<1/8$ 
\cite{RRS1} did not turn up {\em any} regions where the symmetric
bipodal graphon was not optimal. That is not a proof that such regions don't exist, of course, but
it does suggest that such regions are likely to be small. 
\end{itemize}

Despite this evidence, Theorem \ref{main3} says that there is 
a (possibly very small) open subset of the region $e<1/2$, $t<1/8$ on
which the symmetric bipodal graphon is {\em not} optimal. In this section
we state and prove Propositions \ref{rank-1} and \ref{no-bipodal} and 
then prove Theorem \ref{main3}.

\begin{proposition}\label{rank-1} 
Suppose that $e<1/2$ and that $g$ is a graphon of 
the form 
\[ g(x,y) = e + \lambda_1 v_1(x) v_1(y), \]
with edge density $e$ and triangle density $t=e^3 - \sigma^3 < e^3$.
Then $S(g)$ is bounded above by $\frac12 \left (H(e+\sigma) + H(e-\sigma)
\right)$, with equality if and only if $g$ is symmetric bipodal. 
\end{proposition}

\begin{proof} First note that $\int_0^1 v_1(x) \, dx = 0$, since the 
overall edge density is exactly $e$. The triangle density is then
$e^3 + \lambda_1^3$, so $\lambda_1 = -\sigma$. Since $0 \le g(x,y)$, 
$v_1(x)$ is bounded in magnitude by $\sqrt{e/\sigma}$ and 
$\lambda_1 v_1(x) v_2(x)$ is bounded in magnitude by $e$. This implies 
that the power series 
\[ H(g(x,y)) = \sum_{j=0}^\infty \frac{H^{(j)}(e)}{j!} (\lambda_1
v_1(x) v_1(y))^j \]
converges absolutely. Integrating over $x$ and $y$ then gives 
\[ S(g) = \sum_{j=1}^{\infty} 
(-\sigma)^j \frac{H^{(j)}(e)}{j!} \left ( \int_0^1 v_1(x)^j\, dx \right )^2.
\]
Since $e<1/2$, all the odd derivatives of $H$ are positive at $e$, 
while all the even derivatives are negative. Multiplying by 
$(-\sigma)^j$, all of the terms with $j>0$ are negative. We maximize 
the entropy by minimizing $\int_0^1 v_1(x)^j \, dx$ for $j$ even
and by having $\int_0^1 v_1(x)^j \, dx=0$ for $j$ odd. The 
second moment $\int_0^1 v_1(x)^2 \, dx$ is always equal to 1. The 
fourth and higher even moments are minimized when (and only when!) 
$v_1(x)^2$ is constant and equal to 1. Since $\int_0^1 v_1(x) \, dx=0$, this means that 
$v_1(x)$ is $+1$ on a set of measure 1/2 and $-1$ on a set of measure 1/2,
which makes all of the odd integrals zero, as desired. In other words,
the symmetric bipodal graphon is the unique entropy maximizer among
graphons of this form. 
\end{proof}

Before turning to what happens just below the line $t=e^3$, we establish constraints on the form of any 
entropy-maximizing bipodal graphon with $t<e^3$. We use the parameters $(a,b,c,d)$
of Figure \ref{FIG:bipodal} to describe bipodal graphons. 
Without loss of generality we can assume that $a \le b$, since otherwise we could just swap $c$ and $1-c$ while swapping $a$ and $b$. 

\begin{proposition}\label{abcd}
Suppose that $t=e^3 - \sigma^3 < e^3$ and that a graphon $g$ maximizes entropy among all bipodal graphons with 
edge density $e$ and triangle density $t$. 
(We do not assume that $g$ maximizes entropy among all graphons, just that it is the best bipodal graphon.)
Then either $a=b$ and $c=1/2$ (a symmetric bipodal graphon) or $a<b<d$ and $c < 1/2$.
\end{proposition}

\begin{proof}
In this setting, the variational equations (\ref{EL-general}) become
\begin{eqnarray}
H'(a) & = & \Lambda_e + \Lambda_t (ca^2 + (1-c)d^2), \cr 
H'(b) & = & \Lambda_e + \Lambda_t (cd^2 + (1-c)b^2) ,\cr 
H'(d) & = & \Lambda_e + \Lambda_t (cad + (1-c)bd). 
\end{eqnarray}
Subtracting the second equation from the first gives 
\[ H'(a) - H'(b) = \Lambda_t (c(a^2-d^2) + (1-c)(d^2-b^2)). \]
If $a=b$, then the left hand side is zero and the right hand side is a nonzero multiple of $(1-2c)(a^2-d^2)$, implying that
either $c=1/2$ or $a=d$. But if $a=b=d$, then the triangle density is exactly $e^3$, which is a contradiction. Thus $a=b$ implies
that the graphon must by symmetric bipodal. 

We now turn to the possibility that $a<b$. 
The left hand side is then positive, since $H'$ is a decreasing function. Since $\Lambda_t = \frac13 \frac{\partial S}{\partial t}$
is positive, we must have 
\[ c(a^2-d^2) + (1-c)(d^2-b^2) > 0. \]
This either requires $d<a$, in which case $c>1/2$ (since $a^2-d^2 < b^2-d^2$) or $d>b$, in which case $c< 1/2$. 

Let $\lambda_1$ and $\lambda_2$ be the two nonzero eigenvalues of $g$. The trace of $g$ is $\lambda_1+\lambda_2 =ca+(1-c)b$, 
while the trace of $G$ is $\lambda_1^2 + \lambda_2^2 = c^2a^2 + (1-c)^2 b^2 + 2c(1-c)d^2$. From this we can compute 
\[ \lambda_1\lambda_2 = \frac12 ( (\lambda_1+\lambda_2)^2 - (\lambda_1^2+\lambda_2^2)) = c(1-c)(ab-d^2). \] 
If $d$ were less than $a$ and $b$, this would be positive,
meaning that both eigenvalues would be positive. Moreover, one of the two eigenvalues 
is at least $e$, so the triangle density, which is $\lambda_1^3+\lambda_2^3$, would be greater than $e^3$. This rules out the 
possibility that $d<a$, and we conclude that $a<b<d$ and $c<1/2$. 
\end{proof}

When $e > 1/2$ and $t$ is slightly less than $e^3$, the optimal graphon has been proven to take this form, with $a \approx 1-e$,
$b$ slightly less than $e$, and $d$ slightly greater than $e$, and with $c$ small. The situation is different when $e < \frac12$.

\begin{proposition}\label{no-bipodal} 
Suppose that $e<1/2$ and that $g$ is a bipodal graphon
with edge density $e$ and triangle density $t=e^3 - \sigma^3 < e^3$.
Then, for $\sigma$ sufficiently small, $S(g)$ is bounded above by $\frac12 \left (H(e+\sigma) + H(e-\sigma)
\right)$, with equality if and only if $g$ is symmetric bipodal. 
\end{proposition}

\begin{proof} Let 
\[ \Delta a := a-e, \qquad \Delta b := b-e, \qquad \Delta d := d-e \]
and let
\begin{equation} \eta := \frac{c}{1-c} \Delta a + \Delta d, \qquad \alpha := 2 c \eta - \frac{c}{1-c} \Delta a. \end{equation}
The leading term is $\alpha$\footnote{This is not the same as the $\alpha$
in  the proof of Theorem \ref{main1}. There are only so many Greek letters
in the alphabet.}, while $\eta$ measures the extent to which the degree function fails to be constant. 
We can then express all of our quantities in terms of $\alpha$, $\eta$ and $c$.
\begin{eqnarray}\label{Delta-abd}
\Delta a & = & -\frac{1-c}{c} \alpha + e(1-c)\eta, \cr \cr 
\Delta b & = & -\frac{c}{1-c} \alpha -2c \eta, \cr \cr 
\Delta d & = & \alpha + (1-2c)\eta. 
\end{eqnarray}
In terms of these parameters, the triangle density works out to be
\begin{equation}\label{t-count}
e^3-\sigma^3 = \tdens(g) = e^3 + 3(e-\alpha) c(1-c) \eta^2 - \alpha^3. 
\end{equation}
Since the triangle density is less than $e^3$, $\eta^2$ must be 
$O(\alpha^3)$. That is, $\eta$ is much smaller than $\alpha$.

We now compute the entropy
\[S(g) = \sum_{k=0}^\infty \frac{H^{(k)}(e)}{k!} \mu_k, \]
where
\[ \mu_k = \iint (g(x,y)-e)^k \, dx \, dy
= c^2 (\Delta a)^k + (1-c)^2 (\Delta b)^k + 2c(1-c) (\Delta d)^k.\]
Since $e<1/2$, the odd derivatives of $H$ at $e$ are positive, while
the even derivatives are negative, so we 
want to minimize the even moments and maximize the odd moments. 
The symmetric bipodal graphon (uniquely) minimizes the even moments 
and has all the odd moments equal to zero. For an asymmetric bipodal graphon to do as well, it must have some positive odd moments. 

The moment $\mu_k$ is a $k$-th order homogeneous polynomial in
$\alpha$ and $\eta$ with coefficients that depend on $c$. Since 
$\eta =O(\alpha^{3/2})$, 
\[ \mu_k = \left ( 2c(1-c) + (-1)^k \frac{(1-c)^k}{c^{k-2}}
+ (-1)^k \frac{c^k}{(1-c)^{k-2}} \right ) \alpha^k + O(\alpha^{k+\frac12}). \]
The coefficient of $\alpha^k$ is zero when $k=1$, is 1 when $k=2$, is negative 
when $k$ is an odd number greater than 1, and is greater than 1 when $k$ is an even number greater than 2. In other words, all moments with $k>2$ are 
worse, to leading order, than the moments of the symmetric bipodal graphon. 

There is one more point we must account for. For $k$ odd,  
$\left ( 2c(1-c) - \frac{(1-c)^k}{c^{k-2}}
- \frac{c^k}{(1-c)^{k-2}} \right )$ goes to zero as $(1-2c)^2$
as $c \to 1/2$. We must rule out the possibility that other 
contributions to $\mu_k$ might become greater than the $\alpha^k$ 
term as $c$ approaches 1/2. 

This requires estimates on $\eta$. From the formula for the triangle
density, we have that 
\[ \alpha \approx \sigma + \frac{e c(1-c) \eta^2}{\sigma^2},\]
and hence that 
\[ \mu_2 = \alpha^2 + 2c(1-c)\eta^2 \approx \sigma^2 + \frac{2ec(1-c)}{\sigma} \eta^2. \]
That is, there is a cost proportional to $\eta^2/\sigma$ that does 
not vanish as $c \to 1/2$. Meanwhile, all contributions to 
moments involving odd powers of $\eta$ are proportional to $(1-2c)$. 
This is because the graphon is invariant under the transformation
$\eta \to -\eta$, $c \to 1-c$, $a \leftrightarrow b$. The leading
such contribution comes from $\mu_3$ and goes as $(1-2c)\alpha^2\eta
\approx (1-2c) \sigma^2 \eta$ times a polynomial in $c$ that does not 
vanish at $c=1/2$. Setting the derivative of the entropy with respect to
$\eta$ equal to zero tells us that $\eta = O(\sigma^3(1-2c))$. All contributions from odd powers 
of $\eta$ are thus $O(\sigma^5(1-2c)^2)$, and so are dominated by the $\alpha^3$ contribution to $\mu_3$, 
while all contributions from even powers of $\eta$ 
are dominated by the $\mu^2/\sigma$ contribution to $\mu_2$. 

\end{proof}

Thanks to Propositions \ref{rank-1} and \ref{no-bipodal}, any graphon that does better than symmetric bipodal in the region just below
$t=e^3$ with $e<1/2$ must be at least tripodal (or perhaps not even multipodal at all) and the difference between that graphon and a 
constant graphon must have rank at least two. That is exactly what we construct in the proof of Theorem \ref{main3}.

\begin{proof}[Proof of Theorem \ref{main3}]
We consider values of $e$ and $t=e^3-\sigma^3$ where $e<e_0$ and $\sigma$ is sufficiently small. 
The number $e_0$ is defined by the equation
\begin{equation} 3 H'''(e_0)^2 = H''(e_0)H''''(e_0),
\end{equation}
which simplifies to $6e_0^2 - 6e_0 + 1=0$, or $e_0 = (3-\sqrt{3})/6 \approx 0.2113$.
When $e< e_0$, $3H'''(e)^2$ is 
greater than $H''(e)H''''(e)$. In fact, as $e \to 0$, $3 H'''(e)^2$
goes as $3 e^{-4}$, while $H''(e)H''''(e)$ goes as $2 e^{-4}$. However,
as $e$ approaches $1/2$, $3 H'''(e)^2$ goes to zero while $H''(e)H''''(e)$
does not. 

Let $A>B>0$ and let 
\begin{equation}
F(A,B) = \frac{H(e+A+B)+ H(e-A+B) -2H(e) -2BH'(e)}{(A^3-B^3)^{2/3}}.
\end{equation}
We will eventually choose $A$ and $B$ to maximize $F(A,B)$. Pick a 
small number $c$ and divide the interval $[0,1]$ into three pieces:
\[ I_1 = [0, c/2], \qquad I_2 = (c/2,c], \qquad I_3 = (c,1]. \]

Consider the graphon
\begin{equation}
g(x,y) = \begin{cases} e-A+B(1-c) & (x,y) \in I_1 \times I_1 
\cup I_2 \times I_2 \cr 
e+A+B(1-c) & (x,y) \in I_1 \times I_2 \cup I_2 \times I_1 \cr 
e - c B & (x,y) \in [(I_1\cup I_2)\times I_3] \cup [I_3 
\times (I_1 \cup I_2)] \cr 
e + \frac{c^2}{1-c} B & (x,y) \in I_3 \times I_3.
\end{cases}
\end{equation}
Equivalently, $g(x,y) = e - cA v_1(x) v_1(y) + cB v_2(x) v_2(y)$, 
where 
\[ v_1(x) = \begin{cases} c^{-1/2} & x \in I_1, \cr 
- c^{-1/2} & x \in I_2, \cr 
0 & x \in I_3, \end{cases} \qquad 
v_2(x) = \begin{cases} \sqrt{(1-c)/c} & x \in I_1 \cup I_2, \cr 
- \sqrt{c/(1-c)} & x \in I_3. \end{cases}\]
This graphon has edge density $\edens(g)=e$ and triangle density 
\[ \tdens(g) = e^3 - c^3 A^3 + c^3 B^3.\]
Setting the triangle density equal to $e^3-\sigma^3$ gives 
\[ c = \sigma (A^3 - B^3)^{-1/3}. \]

We now estimate the entropy
\begin{eqnarray} 
S(g) &=& \frac{c^2}{2} H\big (e-A+(1-c)B \big ) + \frac{c^2}{2} 
H \big (e+A+(1-c)B \big ) \cr 
&& + 2c(1-c) 
H\big (e- cB\big ) + (1-c)^2 H \left ( e + \frac{c^2}{1-c}B\right )
\end{eqnarray}
to order $c^2$, or equivalently to order $\sigma^2$. 
The first two terms already are $O(c^2)$, so we can simply replace 
$e \pm A + (1-c)B$ with $e \pm A + B$.
For the remaining terms, we can use a linear approximation for $H(u)$. 
The result is
\begin{eqnarray}
S(g) & = & H(e) + \frac{c^2}{2} \left [ H(e-A+B) + H(e+A+B) - 2H(e) -2BH'(e) \right ]
+ O(c^3) \cr \cr 
& = & H(e) + \frac12 F(A,B) \sigma^2 + O(\sigma^3). 
\end{eqnarray}
For comparison, the symmetric bipodal graphon has entropy 
\[ H(e) + \frac12 H''(e) \sigma^2 + O(\sigma^4). \]
If $F(A,B) > H''(e)$, and if $\sigma$ is sufficiently small, then the 
tripodal graphon has more entropy than the symmetric bipodal graphon.

What remains is showing that we can get $F(A,B)> H''(e)$ when 
$e < e_0$. Let $A$ be a small positive number and let 
\[ B = - \frac{H'''(e)}{2 H''(e)} A^2. \]
Since $H'''(e)>0>H''(e)$, $B$ is positive. Since $A^3-B^3 = A^3 + O(A^6)$,
$(A^3 - B^3)^{2/3} = A^2 + O(A^5)$. We compute the numerator of 
$F(A,B)$ to order $A^4$ by doing a 4-th order Taylor series expansion
of $H(e-A+B)$ and $H(e+A+B)$ around $e$ and keeping terms proportional to 
$A^2$, $A^4$, $B$, $B^2$, and $A^2B$. (The expression is even in $A$,
so we only get even powers of $A$.) The result is 
\begin{eqnarray} F(A,B) & = & \frac {(A^2+B^2) H''(e) + A^2B H'''(e)    
+ \frac{1}{12} A^4 H''''(e) + O(A^6)}{(A^3 - B^3)^{2/3}} \cr 
& = & \frac{A^2 H''(e) + A^4 \left ( \frac{H''''(e)}{12} - 
\frac{(H'''(e))^2}{4 H''(e)} \right ) + O(A^6)}{A^2 + O(A^5)} \cr 
& = & H''(e) + \left ( \frac{H''''(e)}{12} - 
\frac{(H'''(e))^2}{4 H''(e)} \right ) A^2 + O(A^3).
\end{eqnarray}
Since $e < e_0$, 
the coefficient of $A^2$ is positive, so $F(A,B)>H''(e)$ when $A$ is small. 
\end{proof}

\section{Symmetry as an order parameter}\label{sec:ordparam}

{\bf Question}: Considering  Figure \ref{FIG:Razborov} and the two open subsets
$\calO$ and $\calO_1$ there defined by Figure \ref{FIG:openset} and
Figure  \ref{FIG:transition}, could $\calO$ and $\calO_1$ 
be part of the same phase?
\medskip

Although there is no barrier between $\calO$ and $\calO_1$ like the
curve
$t=e^3$ between $\calO_2$ and $\calO_1$, the answer must
be ``no'', thanks to the following symmetry argument. On $\calO$, the
entropy-maximizing graphon has constant degree function
$d(x) = \int_0^1 g(x,y) \, dy = e$.  The density $T_2$ of 2-stars is
given by the integral $\int_0^1 d(x)^2 \, dx$, so $Q = T_2 - e^2$ is
identically zero on $\calO$. If  $\calO$ and $\calO_1$ were part of
the same phase and
 we had any analytic curve $c(s)$
running between them, then $Q$ would have to be zero on
the first part of the curve and then by analyticity it would have to be
zero on the entire curve. However, it is easy to check that $Q$ is not
zero on all of $\calO_1$, insofar as the degree function for the graphon of
Theorem \ref{thm:b11} is not constant. Instead, $Q$ is a nonzero
multiple of $\sigma^5$ plus $O(\sigma^6)$. This contradiction proves
our assertion. \qed

\bigskip

This argument has a very similar flavor to an argument that is common
in statistical physics. There, if you can find an ``order parameter''
\cite{A, Ru}, a physical quantity which is identically zero on an open
subset of the parameter space and nonzero in another, then it cannot
be analytic along any path from the first region to the
second, so the open subsets cannot be parts of the same phase.
Finding such an order parameter can be very difficult, but once found it can be
very useful as we now show.

To appreciate the subtleties associated with some phase transitions in
real materials, consider water in various common states. First
consider {\em gaseous} water (steam) at temperature $T_1$ just above
$100^\circ$ Celsius and atmospheric pressure $P$, and {\em liquid}
water at temperature $T_2$ just below $100^\circ$ Celsius and again
pressure $P$.  The mass density is much lower in that gaseous state
than in that liquid state, so in any reasonable sense there is a sharp
transition in state corresponding to the (arbitrarily) small change in
temperature. Now consider the pair of states: {\em liquid} water at
temperature $T_3$ just above $0^\circ$ Celsius and atmospheric
pressure $P$ and {\em solid} water (ice) at temperature $T_4$ just
below $0^\circ$ Celsius and again pressure $P$. Again the mass density
is different between these two states (ice floats on water) so again
in any reasonable sense there is a sharp transition in state corresponding
to the (arbitrarily) small change in temperature.

There is a big
difference between these two transitions. Consider a
different way of changing the state from that gaseous state to its
`neighboring' liquid state. It is expermentally possible, by slowly accessing
high temperatures and high pressures to use a different path between these two states
{\em without making any sharp change in state!} (One says there is a
`critical' point in the gas/liquid transition.) But again
experimentally, there is NO critical point in the liquid/solid
transition; however you vary temperature and pressure to move slowly between
a state of liquid water and a state of solid ice you must go through a sharp transition.
Percy Bridgman received the Nobel prize in 1946 for his
extensive experiments on
high pressure, one result of which was to demonstrate that there is no
critical point on the transition between fluid and solid in {\em any} known
material. It is an old problem to try to understand why this should be
the case. Consider the following quote
in \cite[page 11]{Uh}

\medskip
\begin{itemize}
\item[] The most outstanding unsolved problem of equilibrium statistical
mechanics is the problem of the phase transitions. Why do all
substances occur in at least three phases, the solid, liquid, and
vapor phase which can coexist in the triple point? Why is there, again
for all substances, a critical point for the vapor-liquid equilibrium,
while apparently there is no critical point for the fluid-solid
transition. Note that since these are {\em general} phenomena, they
must have a {\em general} explanation; the precise details of the
molecular structure and of the intermolecular forces should not
matter. 
\end{itemize}

\smallskip

\noindent
As described in \cite[page 19]{A}, Lev Landau tried to use symmetry as 
an order parameter to solve this problem:

\medskip
\begin{itemize}
\item[] It was Landau (1958) who, long ago, first pointed out the vital
importance of symmetry in phase transitions. This, the First Theorem
of solid-state physics, can be stated very simply: it is impossible to
change symmetry gradually. A given symmetry element is either there or
it is not; there is no way for it to grow imperceptibly.  This means,
for instance, that there can be no critical point for the melting
curve as there is for the boiling point: it will never be possible to
go continuously through some high-pressure phase from liquid to
solid.
\end{itemize}

While appealing, Landau's argument is not universally accepted; see
\cite[page 122]{Pi}. Part
of the problem
is that there is no known  model in
equilibrium statistical mechanics which can be {\it proven} to exhibit
both fluid and solid phases \cite{Br,Uh}. Even if such phases were proven to 
exist, it isn't at all clear how to define an appropriate order parameter.

Yet that is exactly what we have done in the context of random graphs 
at the beginning of this section: we
contrasted the subset $\calO$,  part of a ``symmetric'' phase in which
the order parameter $Q$ vanishes identically, with part of a phase in which
$Q$ is not zero.

Of course this is not a solution of the classic problem of 
proving a solid-fluid phase transition in a reasonable statistical
mechanics model; we are working with
random graphs, not with configurations of atoms. But that's actually
the point!  Graph models are a wonderful laboratory for developing,
with full mathematical rigor, techniques that are simply too hard in
statistical physics, a laboratory where important structural questions
can be successfully solved.

\section{Summary}\label{sec:Summ}
This paper is part of a series \cite{RS1,RS2,RRS1,RRS2,RRS3,KRRS1,KRRS2,KRRS3,KKRW,
  KRRS3,RRS3,NRS1,NRS2} studying combinatorial
systems (graphs, permutations, sphere packing) under competing
constraints, as an extension of extremal combinatorics but
concentrating on {\it nonextreme} states of the systems, that is,
states under nonextreme constraints.  We study asymptotically large
systems and for graphs and permutations we use a large deviation
principle (LDP) to analyze ``typical'' ({\em i.e.}\ exponentially most)
constrained states. (We do not know an LDP for sphere packing but
analyze such systems using the hard sphere model \cite{Low} in equilibrium
statistical mechanics.)  In this paper we sharpened our notion of phase by the use of analyticity; see Section \ref{sec:ordparam}. 

Our goal in studying typical nonextremely-constrained states in these
combinatorial systems is to analyze {\em emergent smoothness} response
to infinitesimal change in the constraints, and of the combinatorial
systems we have considered we have found dense graphs the most
amenable to development.

By design our graph modelling has many features in common with that of
the equilibrium statistical mechanics of particles with short range
forces, but has a significant difference: there is no ``distance''
between edges, so each edge has the same influence on any other edge.
In statistical mechanics, models with this feature of the influence of
particles on one another are called ``mean-field'', and although not
part of the mathematical formalism \cite{Ru}, mean-field models such as Curie-Weiss and
van der Waals \cite{Thom} are used to
study phase transitions where more physical models prove too
difficult. The random graph model we have been discussing has, in this
sense, more in common with mean-field models of statistical mechanics,
and, as seen by our success in proving the existence of phases and
phase transitions, may be able to provide a mathematical formalism for studying the
asymptotics of graphs and other combinatorial objects which 
will  be as fruitful mathematically as statistical mechanics has been.

We conclude with the following open problems in this edge/triangle model.

\begin{enumerate}
\item What is the actual structure of the optimal graphon when $e<e_0$ and $t$ is slightly below
$e^3$? Does it resemble the example given in Section \ref{sec:belowER}, or is its structure still wilder?  
How does the behavior of this graphon as $\sigma \to 0$ compare to the moderate
deviations results for $G(n,m)$ in \cite{NRS3}?

\item Is there a succession of phases as $e \to 0$ and $t$ remains close to $e^3$, with tripodal graphons giving way to 4-podal, 5-podal, and so on? 

\item When $e_0 < e < 1/2$ and $t$ is slightly less than $e^3$, is the optimal graphon symmetric bipodal,
or is it something else? 

\item When $e<1/2$ and $t=0$, the optimal graphon is symmetric bipodal. What if $e<1/2$ and $t$ is slightly positive? 

\item Proposition \ref{no-bipodal} is stated for $t$ close to $e^3$. However, the only step in the proof that uses $t
\approx e^3$ is the estimate that $\eta$ is much smaller than $\alpha$. Can the result be extended to the entire region
$e<1/2$, $t<e^3$? 

\item In \cite{KRRS2} it is proven that there are two open sets with
supersatured
triangles, $t > e^3$, which extend to phases. One of these is bounded below by the curve $t=e^3, e<1/2$, while the other
is bounded by $t=e^3, e>1/2$. Are these actually parts of the same phase, or are they distinct? 

\item In Section 4 of \cite{KRRS3}, numerical evidence is given of phase transitions in the 
edge/triangle model along curves where the entropy-optimizing graphon is not unique. 
It would be of interest to prove this. In principle, it may also be
possible to have regions of positive area on which the optimizing graphon is not unique, which
would be a challenge to interpret.
\end{enumerate}

\end{document}